\documentclass{article}
\usepackage{amsmath,amsxtra,amssymb,latexsym,amsfonts,amscd}
\usepackage{multicol,color}
\usepackage{float}
\usepackage{pgf,tikz}
\usepackage{mathrsfs}
\usetikzlibrary{arrows}
\usepackage{amssymb}
\usepackage{theorem}
\usepackage{fancyhdr}
\usepackage{enumerate}
\usepackage[margin=1.1in]{geometry}
\usepackage{bbm}
\usepackage[colorlinks=true,allcolors=blue]{hyperref}
\numberwithin{equation}{section}
\newtheorem{theorem}{Theorem}[section]
\newtheorem{lemma}[theorem]{Lemma}

\newtheorem{proposition}[theorem]{Proposition}
\newtheorem{definition}[theorem]{Definition}
\theoremstyle{plain}{\theorembodyfont{\rmfamily}
}
\theoremstyle{plain}{\theorembodyfont{\rmfamily}
}
\theoremstyle{plain}{\theorembodyfont{\rmfamily}
}
\theoremstyle{plain}{\theorembodyfont{\rmfamily}
\newtheorem{example}[theorem]{Example}}

\theoremstyle{plain}{\theorembodyfont{\rmfamily}
\newtheorem{remark}[theorem]{Remark}}

\newcommand{\gph}{\operatorname{gph}}
\begin{document}
\begin{center}
{\bf OPTIMIZATION OF A PERTURBED IMPLICIT SWEEPING PROCESS\\BY CONSTRAINED DISCONTINUOUS CONTROLS}\\[3ex]
YEN DANG\footnote{Department of Mathematics, San Diego State University, San Diego, California 92182, USA (yen.t.dang01@gmail.com).}
\quad BORIS S. MORDUKHOVICH\footnote{Department of Mathematics, Wayne State University, Detroit, Michigan 48202, USA (boris@math.wayne.edu). Research of this author was partly supported by the US
National Science Foundation under grant DMS-2204519 and by the Australian Research Council under Discovery Project DP250101112.}
\quad DAO NGUYEN\footnote{Department of Mathematics, San Diego State University, San Diego, California 92182, USA.
Corresponding author. E-mail: nntdao1311@gmail.com.
The research of this author was supported by the Department of Mathematics, San Diego State University.}
\end{center}
\small{\sc Abstract.} This paper studies optimal control problems governed by an implicit sweeping process with velocity-dependent constraint geometry, where the dynamics are described by a differential inclusion involving the normal cone to a moving convex polyhedral set. Such systems form a challenging class of nonsmooth differential inclusions due to the implicit coupling between the state, velocity, and constraint reaction. To address these difficulties, we use a method of discrete approximations together with tools of variational analysis and generalized differentiation. This approach allows us to derive necessary optimality conditions for local minimizer in discrete-time problems and subsequently pass to the limit to obtain optimality conditions for the original continuous-time system. The obtained results extend existing optimal control theory for sweeping processes to a class of implicit dynamical systems with state-and velocity-dependent constraints.\\[1ex]
\vspace*{-0.2in}

\normalsize
\section{Introduction and Problem Formulation}\label{sec:intro}

Sweeping processes were introduced by Moreau \cite{Moreau74} in the 1970s to model
elastoplasticity and other mechanical systems with unilateral constraints; see also
\cite{AdlyHaddadThibault14,Brezis73,ColomboThibault10} and the references therein for the
theory of sweeping processes and its developments. Optimal control problems for
sweeping processes have attracted considerable attention during the last decade; see,
e.g., \cite{AdamOutrata14,ArroudColombo18,BrokateKrejci13,CaoMordukhovich16,CaoMordukhovich17,CaoMordukhovich19,CastaingMonteiroRaynaud14,CHHM12,CHHM16,CMN20,dePinhoFerreiraSmirnov19,HoangMordukhovich19,Tolstonogov16}.
In most of these works the sweeping process is \emph{explicit}: the state trajectory is
constrained to remain inside a convex (or prox-regular) moving set, and the
associated normal cone acts directly on the state variable. In such models, the
geometry of the constraint is entirely determined by the position of the state. In
particular, in \cite{CMN20} optimal control of a perturbed explicit sweeping process
by constrained discontinuous controls was studied via the method of discrete
approximations developed in \cite{Mordukhovich95,DonchevFarkhiMordukhovich07}; see also
\cite{BouachHaddadMordukhovich22,BouachHaddadThibault22} for discrete approximations and
optimal control of integro-differential sweeping processes.

Implicit sweeping processes, in which the normal cone is evaluated at a point depending on
the velocity, were studied by Adly and Haddad \cite{AdlyHaddad18} in
connection with quasistatic evolution variational inequalities arising in contact mechanics;
their state-dependent versions and the relationship with evolution quasi-variational
inequalities were investigated in \cite{AdlyHaddad20,AdlyHaddadLe19}. Jourani and Vilches
\cite{JouraniVilches19} developed a differential equation approach to implicit sweeping
processes, showing their equivalence to suitable ordinary differential equations, which
yields existence and uniqueness of solutions.
Well-posedness of history-dependent and state-dependent implicit sweeping processes, together
with applications to quasistatic and viscoelastic contact problems, was established in
\cite{ZengVilches20,Vilches22}. More recently, Alsharief and Jourani \cite{AlshariefJourani25}
used implicit sweeping processes as regularizations of explicit ones in order to derive
optimality conditions for controlled sweeping processes. These works mainly concern
well-posedness of implicit sweeping dynamics or use them as a tool for explicit problems; in
contrast, the present paper addresses optimal control of the implicit sweeping process itself.

In contrast, the present paper is devoted to an \emph{implicit} sweeping process, in
which the contact with the constraint set depends on the system velocity. More
precisely, the normal cone is evaluated at a velocity-dependent contact point of the
form $w(t)=A\dot x(t)+h(t,x(t))$ rather than at the state itself. This feature leads to a
fundamentally different geometric structure and requires new analytical tools in the
derivation of necessary optimality conditions.

We consider the following optimal control problem of Bolza type, denoted by $(P)$:
\begin{equation}
J[x,u]:=\varphi(x(T))
+\int_{0}^{T}\ell\big(t,x(t),u(t),\dot{x}(t)\big)\,dt
\longrightarrow \min,
\label{eq:cost}
\end{equation}
subject to the implicit sweeping process
\begin{equation}
\dot{x}(t)\in
f\big(t,x(t),u(t)\big)
-
N\!\big(C(t);A\dot{x}(t)+h(t,x(t))\big)
\quad \text{a.e. } t\in[0,T],
\label{eq:dynamics}
\end{equation}
with
\begin{equation}
x(0)=x_{0}\in C(0)\subset\mathbb{R}^{n},
\label{eq:initial}
\end{equation}
where $u(t)\in U\subset\mathbb R^d$ for a.e. $t\in[0,T]$ is an admissible control,
$f:[0,T]\times\mathbb{R}^{n}\times U\to\mathbb{R}^{n}$ is a perturbation mapping,
$h:[0,T]\times\mathbb{R}^{n}\to\mathbb{R}^{n}$ is an auxiliary mapping,
$A:\mathbb{R}^{n}\to\mathbb{R}^{n}$ is a bounded linear symmetric operator, and
$C:[0,T]\rightrightarrows\mathbb R^n$ is a set-valued mapping with nonempty, closed,
and convex values.

The moving set $C(t)$ is assumed to be a convex polyhedron represented by
\begin{equation}
C(t):=\bigcap_{j=1}^{s}C^{j}(t),
\qquad
C^{j}(t):=
\Big\{
x\in\mathbb{R}^{n}
\;\Big|\;
\langle x_{*}^{j},x\rangle\le c_{j}(t)
\Big\},
\label{1.4}
\end{equation}
where $x_{*}^{j}\in\mathbb{R}^{n}$ are given outward normals to the $j$-th face of $C(t)$ and
$c_{j}:[0,T]\to\mathbb{R}$ are prescribed functions. The symbol
$N(\cdot;C(t))$ denotes the normal cone to $C(t)$ in the sense of convex analysis.
The normal cone to $C(t)$ at a point $z\in C(t)$ is defined by
\begin{equation}
N(z;C(t))
:=
\left\{
v\in\mathbb{R}^n
\;\middle|\;
\langle v,y-z\rangle\le 0
\quad\forall y\in C(t)
\right\},
\qquad z\in C(t),
\label{1.5}
\end{equation}
and $N(z;C(t))=\emptyset$ whenever $z\notin C(t)$. Following the common usage in the
sweeping process literature, we write both $N(C(t);z)$ and $N(z;C(t))$ for the same
object.
Unlike the classical sweeping process, where the normal cone is evaluated directly at the state, the dynamics \eqref{eq:dynamics} involve the quantity
$A\dot{x}(t)+h(t,x(t))$ inside the normal cone. This implicit dependence creates a velocity-dependent constraint geometry and introduces a strong coupling between the state, velocity, and constraint reaction. As a result, many techniques developed for explicit sweeping processes are no longer directly applicable.

From a modeling perspective, the implicit sweeping process describes the motion of a
system whose velocity is constrained by the geometry of the moving set $C(t)$. Unlike
the explicit sweeping process, where the velocity is determined directly by the normal
cone evaluated at the state, the implicit formulation involves the velocity inside the
normal cone term. This feature models situations in which reaction forces depend not
only on the state but also on the velocity itself, capturing effects such as damping,
feedback, or internal forces. The control $u(\cdot)$ influences the system through the
perturbation mapping $f$, while the moving set $C(t)$ captures time-dependent
constraints.

The main goal of this paper is to derive necessary optimality conditions for local minimizers of problem $(P)$. Our approach combines a discrete approximation scheme with tools of first-order and second-order generalized differentiation. We establish an explicit coderivative characterization of the associated implicit normal-cone mapping, derive necessary optimality conditions for discrete-time approximations, and then pass to the limit to obtain a continuous-time optimality system. The obtained results extend existing optimal control theory for sweeping processes to an implicit dynamical systems with velocity-dependent constraints.

The rest of the paper is organized as follows. Section~\ref{sec:assump} presents the
standing assumptions and basic notions. In Section~\ref{sec:discrete} we construct
discrete approximations of feasible and local optimal solutions to $(P)$ and establish
their strong convergence. Section~\ref{sec:coder} recalls the needed tools of
variational analysis and computes the coderivative of the implicit sweeping mapping.
Section~\ref{sec:discreteNOC} derives necessary optimality conditions for the discrete
implicit sweeping problems, while Section~\ref{sec:continuous} establishes necessary
optimality conditions for the original continuous-time problem by passing to the limit.
Section~\ref{sec:examples} contains examples illustrating the obtained results, and
Section~\ref{sec:concl} gives concluding remarks.

\section{Standing Assumptions and Basic Notions}\label{sec:assump}

For the implicit sweeping dynamics \eqref{eq:dynamics}, we introduce the
velocity-dependent contact point
\begin{equation}
w(t):=A\dot{x}(t)+h(t,x(t)).
\label{2.1}
\end{equation}

Given a point $w(t)\in C(t)$, the corresponding active index set is defined by
\begin{equation}
I(w(t))
:=
\Big\{
j\in\{1,\ldots,s\}
\;\Big|\;
\langle x_*^j,w(t)\rangle=c_j(t)
\Big\}.
\label{2.2}
\end{equation}
We say that the Positive Linear Independence Constraint Qualification (PLICQ) holds at $w(t)$ if
\begin{equation}
\left[
\sum_{j\in I(w(t))}
\alpha_jx_*^j=0,
\quad
\alpha_j\ge0
\right]
\Longrightarrow
\alpha_j=0
\quad
\forall j\in I(w(t)).
\label{2.3}
\end{equation}
If the nonnegativity restriction on $\alpha_j$ is removed in \eqref{2.3}, we obtain the
usual \emph{Linear Independence Constraint Qualification} (LICQ), i.e., the linear
independence of the active generating vectors $\{x_*^j\mid j\in I(w(t))\}$.

Throughout the paper, we impose the following standing assumptions.
\begin{itemize}
\item[(H1)]
The control set $U\subset\mathbb{R}^d$ is nonempty and compact.

\item[(H2)]
The perturbation mapping $f:[0,T]\times\mathbb{R}^n\times U\to\mathbb{R}^n$ is continuous, uniformly Lipschitz continuous in $x$ on bounded sets, and satisfies
\[
\|f(t,x,u)\|
\le
\beta(1+\|x\|)\quad\forall u\in U
\]
for some $\beta>0$.

\item[(H3)]
Along the reference process $(\bar x(\cdot),\bar u(\cdot))$, the PLICQ condition \eqref{2.3} holds at
\[
\bar w(t)
:=
A\dot{\bar x}(t)+h(t,\bar x(t))
\]
for every $t\in[0,T]$.

\item[(H4)]
The operator $A:\mathbb{R}^n\to\mathbb{R}^n$ is symmetric positive definite (coercive), i.e., there exists $\alpha>0$ such that
\[
\langle Av,v\rangle
\ge
\alpha\|v\|^2
\qquad
\forall v\in\mathbb{R}^n.
\]

\item[(H5)]
The state-to-contact shift $h:[0,T]\times\mathbb{R}^n\to\mathbb{R}^n$ is Lipschitz continuous in $x$ on bounded sets, Lipschitz continuous in $t$, and satisfies
\[
\|h(t,x)\|
\le
c(1+\|x\|)
\]
for some constant $c>0$.

\item[(H6)]
For every feasible process $(x(\cdot),u(\cdot))$, the PLICQ condition \eqref{2.3} holds at
\[
w(t)=A\dot x(t)+h(t,x(t))
\]
for all $t\in[0,T]$.
\end{itemize}
In the implicit model, the constraint is enforced at the velocity contact point
\[
w(t)=A\dot x(t)+h(t,x(t))\in C(t).
\]
Under assumptions (H1)--(H6), see \cite{AdlyHaddad18,EdmondThibault05,JouraniVilches19}, for every measurable control $u(\cdot)$ there exists a unique trajectory
$x(\cdot)\in W^{1,2}([0,T];\mathbb R^n)$ satisfying \eqref{eq:dynamics} almost everywhere on $[0,T]$. In particular,
$w(t)=A\dot x(t)+h(t,x(t))\in C(t)$ for a.e. $t\in[0,T]$, and all the associated complementarity conditions are satisfied.

\begin{definition}\label{def:minimizer}
{\rm (i)} A \emph{feasible process} for $(P)$ is a pair $(x(\cdot),u(\cdot))$ satisfying \eqref{eq:dynamics}--\eqref{1.4} together with
\[
w(t)=A\dot x(t)+h(t,x(t))
\in C(t)
\]
for almost every $t\in[0,T]$.

{\rm (ii)} A feasible pair $(\bar x(\cdot),\bar u(\cdot))$ is called a \emph{$W^{1,2}\times L^2$ local minimizer} for $(P)$ if there exists $\varepsilon>0$ such that
\[
J[\bar x,\bar u]
\le
J[x,u]
\]
for all feasible pairs satisfying
\begin{equation}
\int_0^T
\Big(
\|\dot x(t)-\dot{\bar x}(t)\|^2
+
\|u(t)-\bar u(t)\|^2
\Big)\,dt
<\varepsilon.
\label{2.loc}
\end{equation}

{\rm (iii)} Similarly, $(\bar x(\cdot),\bar u(\cdot))$ is a \emph{relaxed $W^{1,2}\times L^2$ local minimizer} for $(P)$ if
\[
\varphi(\bar x(T))\le\varphi(x(T))
\]
whenever $(x(\cdot),u(\cdot))$ is feasible for the relaxed implicit problem (in which the
control set $U$ is replaced by its convex hull, cf.\ \cite{EdmondThibault05}) and satisfies the
same closeness condition \eqref{2.loc} in $W^{1,2}\times L^2$.
\end{definition}

We define the set-valued mapping
\begin{equation}
F(t,x,u,v)
:=
N(C(t);Av+h(t,x))
-f(t,x,u),
\label{2.4}
\end{equation}
where $v=\dot x$. Using the polyhedral representation of $C(t)$, the mapping $F$ admits the equivalent form
\begin{equation}
F(t,x,u,v)
=
\left\{
\sum_{j\in I(w)}
\lambda_jx_*^j
\;\middle|\;
\lambda_j\ge0
\right\}
-f(t,x,u),
\label{2.5}
\end{equation}
where $w=Av+h(t,x)$.

\section{Discrete Approximation of the Implicit Sweeping Process}\label{sec:discrete}

To derive necessary optimality conditions, we approximate the continuous-time problem by a sequence of finite-dimensional optimization problems. The approximation is obtained through a forward Euler discretization of the implicit sweeping dynamics. We first introduce the discrete inclusion, then establish existence of discrete optimal solutions and show that the resulting approximations converge to the corresponding feasible and optimal solutions of the original problem. We analyze problem $(P)$ by discretizing the time derivative via the forward Euler scheme and then passing to the limit as the mesh size tends to zero.

For each \(m\in\mathbb N\), consider the uniform partition
\[
\Delta_m:=\{0=t_m^0<t_m^1<\cdots<t_m^{2^m}=T\},
\qquad
h_m:=\frac{T}{2^m},
\qquad
t_m^i:=ih_m .
\]
The discrete velocity is defined by
\[
v_m^i:=\frac{x_m^{i+1}-x_m^i}{h_m},
\]
and the implicit sweeping dynamics are approximated through the inclusion
\begin{equation}
v_m^i
\in
f(t_m^i,x_m^i,u_m^i)
-
N\!\left(
C(t_m^i);
A\frac{x_m^{i+1}-x_m^i}{h_m}
+h(t_m^i,x_m^i)
\right),
\label{3.disc}
\end{equation}
for \(i=0,\ldots,2^m-1\), with initial condition \(x_m^0=x_0\) and controls
\(u_m^i\in U\). For convenience, we denote $I_m^i:=[t_m^{i-1},t_m^i)$ for $i=1,\ldots,2^m$,
and let $\mathbb B$ stand for the closed unit ball of the space in question.

Given a feasible reference pair $(\bar x(\cdot),\bar u(\cdot))$ and radii
$\rho_x,\rho_u>0$, consider the compact box
\begin{equation}
K_m:=\prod_{i=0}^{2^m}B\big(\bar x(t_m^i),\rho_x\big)\times\prod_{i=0}^{2^m-1}B\big(\bar u(t_m^i),\rho_u\big),
\label{3.box}
\end{equation}
where $B(\bar x(t_m^i),\rho_x)$ denotes the closed ball in $\mathbb R^n$ centered at
$\bar x(t_m^i)$ with radius $\rho_x$ (and similarly for the control balls). For
$y:=(x^0,\ldots,x^{2^m},u^0,\ldots,u^{2^m-1})\in\mathbb R^{(2^m+1)n}\times\mathbb R^{2^md}$,
the \emph{localized feasible set} is
\begin{equation}
F_m:=\big\{y\in K_m\;\big|\;y\ \text{satisfies the discrete constraints \eqref{3.cons}}\big\},
\label{3.Fm}
\end{equation}
where the discrete constraints are
\begin{equation}
\begin{cases}
x^{i+1}-x^i=h_mv^i,\\[1mm]
v^i\in f(t_m^i,x^i,u^i)-N\big(C(t_m^i);Av^i+h(t_m^i,x^i)\big),\\[1mm]
u^i\in U,\qquad x^0=\bar x(0),\qquad i=0,\ldots,2^m-1.
\end{cases}
\label{3.cons}
\end{equation}
The \emph{localized discrete implicit problem} $(P_m)$ consists of minimizing the cost
\begin{equation}
J_m[x_m,u_m]:=\varphi(x_m^{2^m})+h_m\sum_{i=0}^{2^m-1}\ell(t_m^i,x_m^i,u_m^i,v_m^i)
+\frac12\int_0^T\Big(\|\dot x_m(t)-\dot{\bar x}(t)\|^2+\|u_m(t)-\bar u(t)\|^2\Big)dt
\label{3.Jm}
\end{equation}
over $y=(x_m,u_m)\in F_m$, where $x_m(\cdot)$ and $u_m(\cdot)$ denote the piecewise linear
and piecewise constant extensions of the discrete states and controls to $[0,T]$,
respectively. The quadratic term in \eqref{3.Jm} is a localization penalty measuring the
distance to the reference process; it guarantees strong convergence of discrete
minimizers.

\begin{proposition}\label{prop:exist}
Assume \((H1)\)--\((H6)\) and that $\varphi:\mathbb R^n\to\mathbb R$ and the running cost $\ell$ are lower semicontinuous. Then, for all sufficiently large \(m\in\mathbb N\), the localized discrete problem $(P_m)$ admits at least one optimal solution.
\end{proposition}

\begin{proof}
\ Fix $m\in\mathbb N$ and set $h_m:=T/2^m$, $t_m^i:=ih_m$ for $i=0,\ldots,2^m$.
By the density of piecewise constant controls and standard approximation results for
sweeping processes (see \cite{EdmondThibault05} and Theorem~\ref{thm:approx} below),
there exists $m_0\in\mathbb N$ such that for all $m\ge m_0$ one can construct a
feasible point $(x^*,u^*)\in K_m$ satisfying \eqref{3.cons}. Hence $F_m\ne\emptyset$ for
all $m\ge m_0$.

Let us show that $F_m$ is closed. Take $y_k=(x_k,u_k)\in F_m$ with $y_k\to y=(x,u)$ as
$k\to\infty$. Since $K_m$ is closed, $y\in K_m$. We verify that $y$ satisfies \eqref{3.cons}:
\begin{itemize}
\item The equalities $x^{i+1}-x^i=h_mv^i$ define the zero set of a continuous linear
mapping, and hence this constraint set is closed.
\item The membership $u^i\in U$ is closed by (H1), and thus passes to the limit.
\item For the inclusion, rewrite it as
\[
v^i-f(t_m^i,x^i,u^i)\in-N\big(C(t_m^i);Av^i+h(t_m^i,x^i)\big).
\]
By (H2) and (H5), $f$ and $h$ are continuous, and $A$ is linear by (H4). For each fixed
$i$, define
\[
v_k^i:=\frac{x_k^{i+1}-x_k^i}{h_m},\qquad
w_k^i:=Av_k^i+h(t_m^i,x_k^i),\qquad
\xi_k^i:=f(t_m^i,x_k^i,u_k^i)-v_k^i .
\]
Since $y_k\to y$, we have $v_k^i\to v^i$, and by the continuity of $f$ and $h$,
\[
w_k^i\to w^i:=Av^i+h(t_m^i,x^i),\qquad
\xi_k^i\to\xi^i:=f(t_m^i,x^i,u^i)-v^i .
\]
Because $\xi_k^i\in N(C(t_m^i);w_k^i)$ and the graph of the normal cone mapping to the
closed convex set $C(t_m^i)$ is closed, it follows that $\xi^i\in N(C(t_m^i);w^i)$.
Therefore,
\[
v^i\in f(t_m^i,x^i,u^i)-N\big(C(t_m^i);Av^i+h(t_m^i,x^i)\big).
\]
\end{itemize}
Since $v^i=(x^{i+1}-x^i)/h_m$ is affine in $(x^i,x^{i+1})$, the inclusion is preserved in
the limit. Hence $y\in F_m$, and $F_m$ is closed. Since $K_m$ is a finite product of
compact sets (recall that $U$ is compact and that closed balls are compact in finite
dimensions), it is compact. Therefore, $F_m\subset K_m$ being closed is compact.

Consider now the cost \eqref{3.Jm}. The localization penalty is continuous on $K_m$.
Since $v^i=(x^{i+1}-x^i)/h_m$ is an affine function of $(x^i,x^{i+1})$, compositions of
lower semicontinuous functions with continuous mappings preserve lower
semicontinuity, and finite sums of lower semicontinuous functions are lower
semicontinuous. Hence $J_m$ is lower semicontinuous on $K_m$, and therefore on $F_m$.
For each $m\ge m_0$, the lower semicontinuous function $J_m$ is defined on the nonempty
compact set $F_m$, and by the Weierstrass extreme value theorem $J_m$ attains its minimum on
$F_m$. Thus $(P_m)$ admits at least one optimal solution.
\end{proof}

The next result establishes that feasible trajectories of the continuous-time problem can be approximated by discrete feasible trajectories generated by the Euler scheme.

\begin{theorem}\label{thm:approx}
Let \((\bar x(\cdot),\bar u(\cdot))\) be a feasible pair for \((P)\), where
\(\bar u(\cdot)\) has bounded variation and admits a right-continuous representative on \([0,T]\). Under assumptions \((H1)\)--\((H6)\), for each $m\in\mathbb N$ there exist state-control pairs \((x_m(\cdot),u_m(\cdot))\) and perturbation functions $r_m(\cdot),\rho_m(\cdot)\in\mathbb B$ such that:
\begin{itemize}
\item[\rm(a)] $u_m:[0,T]\to U$ is constant on every $I_m^i$, and
\[
u_m \to \bar u
\quad \text{strongly in } L^2([0,T];\mathbb R^d)\ \text{and pointwise on }[0,T];
\]
\item[\rm(b)] $x_m:[0,T]\to\mathbb R^n$ is affine on each $I_m^i$, satisfies $x_m(0)=x_0$, $x_m(t_m^i)=\bar x(t_m^i)$, and
\[
x_m \to \bar x
\quad \text{strongly in } W^{1,2}([0,T];\mathbb R^n);
\]
\item[\rm(c)] for all $t\in(t_m^{i-1},t_m^i)$ and $i=1,\ldots,2^m$, we have
\[
\dot x_m(t)\in f\big(t_m^i,x_m(t_m^i),u_m(t)\big)-N\Big(C(t_m^i);A\dot x_m(t)+h\big(t_m^i,x_m(t_m^i)\big)\Big)+r_m(t)\rho_m(t),
\]
where $r_m,\rho_m$ are piecewise constant on $I_m^i$, and
\[
r_m(\cdot)\to0\quad\text{in }L^2(0,T)\ \text{ as }m\to\infty .
\]
\end{itemize}
In particular, the interpolated trajectories satisfy a perturbed differential inclusion whose perturbation vanishes as \(m\to\infty\).
\end{theorem}

\begin{proof}
\ Fix $m\in\mathbb N$, let $h_m:=T/2^m$ and $t_m^i:=ih_m$, and recall that
$I_m^i:=[t_m^{i-1},t_m^i)$, $i=1,\ldots,2^m$. By Proposition~\ref{prop:exist} and the closed
graph property of the mapping $w\mapsto N(C(t);w)$, there exists $m_0\in\mathbb N$ such that
$F_m\ne\emptyset$ for all $m\ge m_0$. For each $m\ge m_0$, select a feasible point
\[
y_m=(x_m^0,\ldots,x_m^{2^m},u_m^0,\ldots,u_m^{2^m-1})\in F_m .
\]
We use these feasible discrete points to define piecewise constant controls and piecewise
affine state interpolants. Define $u_m:[0,T]\to U$ by
\[
u_m(t):=u_m^i\quad\text{for }t\in I_m^i .
\]
Since $\bar u$ has bounded variation and is right continuous, the step function
approximations along refining uniform partitions satisfy $u_m\to\bar u$ strongly in
$L^2([0,T];\mathbb R^d)$ and pointwise at every continuity point of $\bar u$.
Next, define the state interpolant $x_m:[0,T]\to\mathbb R^n$ by
\[
x_m(t):=x_m^{i-1}+\frac{t-t_m^{i-1}}{h_m}\big(x_m^i-x_m^{i-1}\big)
=x_m^{i-1}+(t-t_m^{i-1})v_m^{i},\qquad t\in I_m^i ,
\]
where, in this proof, $v_m^i:=(x_m^i-x_m^{i-1})/h_m$. Then $x_m(t_m^i)=x_m^i$ and
$\dot x_m(t)=v_m^{i}$ for a.e. $t\in I_m^i$. Set
\[
w_m^i:=Av_m^i+h(t_m^i,x_m^i).
\]
From the discrete inclusion we have
\[
v_m^i\in f(t_m^i,x_m^i,u_m^i)-N\big(C(t_m^i);w_m^i\big),\qquad i=1,\ldots,2^m,
\]
and hence
\[
\dot x_m(t)\in f(t_m^i,x_m^i,u_m^i)-N\big(C(t_m^i);A\dot x_m(t)+h(t_m^i,x_m^i)\big)
\quad\text{for a.e. }t\in I_m^i .
\]
By the discrete approximation and well-posedness results for the implicit sweeping
process (see \cite{EdmondThibault05}), the interpolating trajectories converge to the
reference feasible pair. More precisely,
\[
x_m\to\bar x\ \text{ strongly in }W^{1,2}([0,T];\mathbb R^n),\qquad
u_m\to\bar u\ \text{ strongly in }L^2([0,T];\mathbb R^d).
\]
This proves all the three statements.
\end{proof}

The approximation theorem yields compactness properties that are crucial in the passage to the limit.

\begin{lemma}\label{lem:bound}
Suppose that assumptions \((H1)\)--\((H6)\) are satisfied, and let
$(x^0,\ldots,x^{2^m},u^0,\ldots,u^{2^m-1})\in F_m$ satisfy the discrete inclusion
\begin{equation}
v_m^i
\in
f(t_m^i,x^i,u^i)
-
N\!\left(C(t_m^i);Av_m^i+h(t_m^i,x^i)\right),
\qquad i=1,\ldots,2^m,
\label{3.lem}
\end{equation}
where $v_m^i:=(x^i-x^{i-1})/h_m$ and $h_m=t_m^i-t_m^{i-1}$.
Then the corresponding discrete velocities are uniformly bounded in
\(L^2\). More precisely, there exists a constant
\(C>0\), independent of the discretization level \(m\), such that
\[
\sum_{i=1}^{2^m}\|v_m^i\|^2h_m\le C.
\]
\end{lemma}

\begin{proof}
\ By Theorem~\ref{thm:approx}, we have $x_m\to\bar x$ strongly in $W^{1,2}([0,T];\mathbb R^n)$.
Hence the sequence $\{x_m\}$ is bounded in $W^{1,2}([0,T];\mathbb R^n)$, and so there exists
a constant $C>0$, independent of $m$, such that
\[
\int_0^T\|\dot x_m(t)\|^2\,dt\le C\quad\text{for all }m .
\]
Since $\dot x_m(t)=v_m^i$ for a.e. $t\in I_m^i$, we get
\[
\int_0^T\|\dot x_m(t)\|^2\,dt=\sum_{i=1}^{2^m}\int_{I_m^i}\|v_m^i\|^2\,dt=\sum_{i=1}^{2^m}\|v_m^i\|^2h_m .
\]
Therefore,
\[
\sum_{i=1}^{2^m}\|v_m^i\|^2h_m\le C,
\]
which completes the proof.
\end{proof}

To obtain convergence of optimal solutions, we use the localized cost functional \eqref{3.Jm}, which penalizes deviations from a reference minimizer. This localization guarantees strong convergence of discrete minimizers.

\begin{theorem}\label{thm:optconv}
Assume that $(\bar x(\cdot),\bar u(\cdot))$ is a relaxed
$W^{1,2}\times L^2$ local minimizer of problem $(P)$, and suppose that
the terminal cost $\varphi$ is lower semicontinuous and continuous at
$\bar x(T)$.
Let
$(\bar x_m,\bar u_m)$
be any sequence of optimal solutions of the localized discrete
approximation problems $(P_m)$, extended piecewise linearly and piecewise constantly to $[0,T]$, respectively, without relabeling. Then
\begin{equation}
\lim_{m\to\infty}
\int_0^T
\left(
\|\dot{\bar x}_m(t)-\dot{\bar x}(t)\|^2
+
\|\bar u_m(t)-\bar u(t)\|^2
\right)\,dt
=0.
\label{3.conv}
\end{equation}
Consequently,
\[
(\bar x_m,\bar u_m)
\longrightarrow
(\bar x,\bar u)
\]
strongly in
\[
W^{1,2}([0,T];\mathbb R^n)
\times
L^2([0,T];\mathbb R^d).
\]
Moreover,
\[
\bar x_m(T)\rightarrow\bar x(T),
\qquad
\varphi(\bar x_m(T))
\rightarrow
\varphi(\bar x(T)).
\]
\end{theorem}

\begin{proof}
\ It is sufficient to verify \eqref{3.conv}. Assume to the contrary that there exists a
subsequence, not relabeled, such that
\[
\gamma_m:=\int_0^T\Big(\|\dot{\bar x}_m(t)-\dot{\bar x}(t)\|^2+\|\bar u_m(t)-\bar u(t)\|^2\Big)dt\to\gamma>0 .
\]
Since $(\bar x_m,\bar u_m)$ is an optimal solution of the localized discrete problem
$(P_m)$, the localization constraint yields the boundedness of the sequence
$\{(\bar x_m,\bar u_m)\}$ in $W^{1,2}([0,T];\mathbb R^n)\times L^2([0,T];\mathbb R^d)$. Hence,
passing to a subsequence if necessary, we may assume that
\[
\bar u_m\rightharpoonup\tilde u\ \text{ in }L^2([0,T];\mathbb R^d),\qquad
\dot{\bar x}_m\rightharpoonup\tilde v\ \text{ in }L^2([0,T];\mathbb R^n).
\]
Define
\[
\tilde x(t):=x_0+\int_0^t\tilde v(s)\,ds,\qquad t\in[0,T].
\]
Then $\tilde x\in W^{1,2}([0,T];\mathbb R^n)$ and $\dot{\tilde x}(t)=\tilde v(t)$ for a.e.
$t\in[0,T]$. By Mazur's theorem, there exists a sequence of convex combinations of
$\{(\dot{\bar x}_m,\bar u_m)\}$ converging strongly to $(\dot{\tilde x},\tilde u)$ in
$L^2([0,T];\mathbb R^n)\times L^2([0,T];\mathbb R^d)$. Therefore, along a subsequence, these
convex combinations converge pointwise almost everywhere on $[0,T]$. Since
$\bar u_m(t)\in U$ a.e. and the relaxed problem is considered with the convexified
control set, we obtain that $\tilde u(t)$ belongs to the relaxed control set for a.e.
$t\in[0,T]$.

Next, since each $(\bar x_m,\bar u_m)$ is feasible for $(P_m)$, its piecewise
linear--piecewise constant extension satisfies, for a.e. $t\in I_m^i$,
\[
\dot{\bar x}_m(t)\in f\big(t_m^i,\bar x_m(t_m^i),\bar u_m(t)\big)
-N\Big(C(t_m^i);A\dot{\bar x}_m(t)+h\big(t_m^i,\bar x_m(t_m^i)\big)\Big).
\]
Using the convergence properties of the interpolants, the continuity of $f$ and $h$, and
the closed-graph property of the normal cone mapping, we pass to the limit as
$m\to\infty$ and conclude that $(\tilde x,\tilde u)$ is feasible for the relaxed implicit
sweeping problem.

Now apply Theorem~\ref{thm:approx} to the reference feasible pair $(\bar x,\bar u)$. Then
there exists a sequence of discrete feasible pairs $(x_m,u_m)$ for $(P_m)$ such that, after
the piecewise linear and piecewise constant extensions to $[0,T]$,
\[
x_m\to\bar x\ \text{ strongly in }W^{1,2}([0,T];\mathbb R^n),\qquad
u_m\to\bar u\ \text{ strongly in }L^2([0,T];\mathbb R^d).
\]
In particular, by the continuity of the localization penalties and of $\varphi$ at
$\bar x(T)$, together with the lower semicontinuity of $\ell$, we have
\[
J_m[x_m,u_m]\to J[\bar x,\bar u].
\]
Since $(\bar x_m,\bar u_m)$ is optimal for $(P_m)$, it follows that
\[
J_m[\bar x_m,\bar u_m]\le J_m[x_m,u_m]\quad\text{for all }m .
\]
Passing to the lower limit on the left-hand side and using lower semicontinuity, we obtain
\[
J[\tilde x,\tilde u]\le\liminf_{m\to\infty}J_m[\bar x_m,\bar u_m]\le\limsup_{m\to\infty}J_m[x_m,u_m]=J[\bar x,\bar u].
\]
On the other hand, since $(\bar x,\bar u)$ is a relaxed $W^{1,2}\times L^2$ local minimizer,
there exists $\varepsilon>0$ such that for all relaxed feasible $(x,u)$ with
\[
\int_0^T\Big(\|\dot x(t)-\dot{\bar x}(t)\|^2+\|u(t)-\bar u(t)\|^2\Big)dt<\varepsilon
\]
we have $J[\bar x,\bar u]\le J[x,u]$. Since $\gamma>0$, the sequence $(\bar x_m,\bar u_m)$
stays away from $(\bar x,\bar u)$. Passing to the limit along the subsequence constructed
above, we obtain a feasible pair $(\tilde x,\tilde u)$ in the same admissible neighborhood
but distinct from $(\bar x,\bar u)$ such that $J[\tilde x,\tilde u]\le J[\bar x,\bar u]$, and
therefore $J[\tilde x,\tilde u]=J[\bar x,\bar u]$. This contradicts the
local minimality of $(\bar x,\bar u)$. Hence $\gamma=0$, and \eqref{3.conv} holds, i.e.,
\[
(\bar x_m,\bar u_m)\to(\bar x,\bar u)\ \text{ in }W^{1,2}([0,T];\mathbb R^n)\times L^2([0,T];\mathbb R^d).
\]
The convergence of terminal states $\bar x_m(T)\to\bar x(T)$ follows immediately.
Since $\varphi$ is continuous at $\bar x(T)$, we also obtain
$\varphi(\bar x_m(T))\to\varphi(\bar x(T))$, which completes the proof.
\end{proof}

\begin{remark}\label{rem:consistency}
Theorem~\ref{thm:approx} constructs piecewise linear state approximations and piecewise
constant controls from discrete feasible solutions. The corresponding interpolated
trajectories satisfy a perturbed differential inclusion, where the perturbation vanishes as
the mesh size tends to zero. By compactness arguments and the well-posedness result of
Edmond and Thibault \cite{EdmondThibault05}, these discrete trajectories converge strongly in
$W^{1,2}([0,T];\mathbb R^n)$ to the reference trajectory. Moreover,
Theorem~\ref{thm:optconv} ensures that optimal solutions of the discrete problems converge
strongly to a local minimizer of the original problem. This establishes the consistency
of the discrete approximation scheme for the implicit sweeping process.
\end{remark}

The results of this section show that the Euler discretization preserves both feasibility and optimality in the limit. Therefore, the continuous-time implicit sweeping control problem can be studied through a sequence of finite-dimensional approximations. This discrete framework will serve as the basis for deriving necessary optimality conditions in the next sections.

\section{Tools of Variational Analysis and Coderivative of the Implicit Sweeping Mapping}\label{sec:coder}

\subsection{Generalized differential constructions}

To derive necessary optimality conditions for the discrete problems $(P_m)$ and for $(P)$
via a limiting procedure, it is essential to employ suitable generalized differential
constructions of variational analysis. These constructions should possess robustness,
admit comprehensive calculus rules, and ensure that the resulting normal cone is not
excessively large when applied to structured graphical sets. This requirement is not
satisfied by the Clarke normal cone, which is a linear subspace of maximal dimension for
sets that are graphically homeomorphic to graphs of Lipschitz functions; see
\cite{Mordukhovich06I,RockafellarWets98} for further discussion. For instance, the Clarke normal
cone to the graph of the simplest convex function $|x|$ at the origin is the whole plane
$\mathbb R^2$. All the properties needed below are ensured by the generalized
differential constructions developed by Mordukhovich; see
\cite{Mordukhovich06I,Mordukhovich06II,Mordukhovich18,RockafellarWets98,Vinter00} for the
first-order theory and applications, and \cite{Mordukhovich06II,Mordukhovich18} for the
second-order constructions used below.

Recall first the (Painlev\'e--Kuratowski) outer limit of a set-valued mapping
$F:\mathbb R^n\rightrightarrows\mathbb R^m$ at $\bar x$ with $F(\bar x)\ne\emptyset$:
\begin{equation}
\mathop{\rm Limsup}_{x\to\bar x}F(x):=\big\{y\in\mathbb R^m\;\big|\;\exists\,x_k\to\bar x,\ y_k\to y\ \text{ with }y_k\in F(x_k),\ k\in\mathbb N\big\}.
\label{4.Limsup}
\end{equation}
Let $\Omega\subset\mathbb R^n$ be locally closed around $\bar x\in\Omega$. The \emph{basic
(limiting, Mordukhovich) normal cone} to $\Omega$ at $\bar x$ is defined by
\begin{equation}
N(\bar x;\Omega)=N_\Omega(\bar x):=\mathop{\rm Limsup}_{x\to\bar x}\big[{\rm cone}\big(x-\Pi(x;\Omega)\big)\big],
\label{4.normal}
\end{equation}
where $\Pi(x;\Omega):=\{u\in\Omega\mid\|x-u\|={\rm dist}(x;\Omega)\}$ is the Euclidean
projection of $x$ onto $\Omega$, and ``cone'' stands for the conic hull. In the convex
case, \eqref{4.normal} coincides with the normal cone \eqref{1.5} of convex analysis,
while in general it may be nonconvex.

Given a set-valued mapping $F:\mathbb R^n\rightrightarrows\mathbb R^m$ with domain and graph
\[
{\rm dom}\,F:=\{x\in\mathbb R^n\mid F(x)\ne\emptyset\},\qquad
\gph F:=\{(x,y)\in\mathbb R^n\times\mathbb R^m\mid y\in F(x)\},
\]
which is locally closed around $(\bar x,\bar y)\in\gph F$, the \emph{coderivative} of $F$ at
$(\bar x,\bar y)$ is defined via \eqref{4.normal} by
\begin{equation}
D^*F(\bar x,\bar y)(u):=\big\{v\in\mathbb R^n\;\big|\;(v,-u)\in N\big((\bar x,\bar y);\gph F\big)\big\},\qquad u\in\mathbb R^m .
\label{4.coder}
\end{equation}
If $F:\mathbb R^n\to\mathbb R^m$ is single-valued and continuously differentiable around
$\bar x$, then
\[
D^*F(\bar x)(u)=\{\nabla F(\bar x)^*u\}\quad\text{for all }u\in\mathbb R^m,
\]
where $\nabla F(\bar x)^*$ is the adjoint (transpose) of the Jacobian matrix and
$\bar y=F(\bar x)$ is omitted from the notation.

Let $\phi:\mathbb R^n\to\overline{\mathbb R}:=(-\infty,\infty]$ be lower semicontinuous with
${\rm dom}\,\phi:=\{x\mid\phi(x)<\infty\}$ and ${\rm epi}\,\phi:=\{(x,\alpha)\in\mathbb R^{n+1}\mid\alpha\ge\phi(x)\}$.
The (first-order, limiting) \emph{subdifferential} of $\phi$ at $\bar x\in{\rm dom}\,\phi$ is
defined geometrically by
\begin{equation}
\partial\phi(\bar x):=\big\{v\in\mathbb R^n\;\big|\;(v,-1)\in N\big((\bar x,\phi(\bar x));{\rm epi}\,\phi\big)\big\}.
\label{4.subdiff}
\end{equation}
This definition admits equivalent analytic descriptions; see
\cite{Mordukhovich06I,RockafellarWets98}. In particular, $N(\bar x;\Omega)=\partial\delta_\Omega(\bar x)$
for $\bar x\in\Omega$, where $\delta_\Omega$ is the indicator function of $\Omega$, equal to $0$ on $\Omega$
and $+\infty$ outside of it. Given $\bar v\in\partial\phi(\bar x)$, the \emph{second-order
subdifferential} (generalized Hessian) of $\phi$ at $(\bar x,\bar v)$ is
\begin{equation}
\partial^2\phi(\bar x,\bar v)(u):=(D^*\partial\phi)(\bar x,\bar v)(u),\qquad u\in\mathbb R^n .
\label{4.second}
\end{equation}
If $\phi$ is $\mathcal C^2$ around $\bar x$, then $\partial^2\phi(\bar x,\bar v)(u)=\{\nabla^2\phi(\bar x)u\}$
with $\bar v=\nabla\phi(\bar x)$. For $\phi=\delta_\Omega$ we get
$\partial^2\delta_\Omega(\bar x,\bar v)(u)=(D^*N_\Omega)(\bar x,\bar v)(u)$ for $\bar v\in N(\bar x;\Omega)$.

This second-order construction is evaluated below for the polyhedral sets $C(t)$ in
\eqref{1.4}. Given a contact point $w\in C(t)$ with the active index set $I(w)$ from
\eqref{2.2} and a direction $p\in\mathbb R^n$, we use the index sets
\begin{equation}
I_0(t,p):=\big\{j\in I(w)\;\big|\;\langle x_*^j,p\rangle=0\big\},\qquad
I_>(t,p):=\big\{j\in I(w)\;\big|\;\langle x_*^j,p\rangle>0\big\},
\label{4.index}
\end{equation}
which describe, by \cite[Theorems~4.5 and 4.6]{HenrionMordukhovichNam10}, the coderivative
of the normal cone mapping to a convex polyhedron.

\subsection{Coderivative of the implicit sweeping mapping}

We now develop the variational tool needed to derive the discrete necessary optimality
conditions. In contrast to explicit sweeping dynamics, the implicit sweeping model
involves velocity-dependent contact with the moving constraint set. As a consequence,
the discrete-time constraints associated with the implicit dynamics depend not only on
the state and control variables, but also on the discrete velocities.

Passing to the discrete-time setting, let $v_i$ denote the discrete velocity
approximating $\dot x(t_i)$. At each grid point $t_i$, the implicit sweeping inclusion
can be equivalently rewritten in the zero-inclusion form
\begin{equation}
0\in v_i-f(t_i,x_i,u_i)+N\big(C(t_i);Av_i+h(t_i,x_i)\big).
\label{4.zero}
\end{equation}
This formulation shows that the discrete constraint acts on the triple $(x_i,v_i,u_i)$.
Consequently, the analysis of the implicit sweeping process requires computing the
coderivative of a set-valued mapping depending simultaneously on the state, velocity,
and control variables. Thus, for each fixed $t\in[0,T]$, we consider the set-valued mapping
\begin{equation}
G_t(x,v,u):=v-f(t,x,u)+N(C(t);Av+h(t,x)),
\label{Gt}
\end{equation}
acting on $(x,v,u)\in\mathbb R^n\times\mathbb R^n\times\mathbb R^d$. This mapping replaces the
explicit sweeping mapping
\begin{equation}
F(x,u):=N(x;C)-g(x,u)
\label{4.explicit}
\end{equation}
used in the explicit setting \cite{CMN20}, and its graphical structure dictates the
variational constructions needed for the implicit sweeping process.

Set \(\xi:=(x,v,u)\) and define
\begin{equation}
g(\xi):=v-f(t,x,u),\qquad
\phi(\xi):=Av+h(t,x),\qquad
F(w):=N(C(t);w).
\label{decomp}
\end{equation}
Then
\begin{equation}
G_t(\xi)=g(\xi)+F(\phi(\xi)).
\label{Gtdecomp}
\end{equation}
This decomposition allows us to compute the coderivative of $G_t$ by combining the
coderivative sum and chain rules with the known coderivative formula for the normal cone
mapping to a polyhedral set. In particular, the coderivative of $G_t$ can be expressed in
terms of the coderivative of the normal cone mapping to $C(t)$ evaluated at the contact
point $w=\phi(\xi)$.

\begin{theorem}\label{thm:coder}
For a fixed time $t\in[0,T]$, consider the set-valued mapping
\[
G_t(x,v,u):=v-f(t,x,u)+N(C(t);Av+h(t,x)).
\]
Suppose that $(x,v,u,\omega)\in\gph G_t$, so that
\[
\omega+f(t,x,u)\in N(C(t);Av+h(t,x)).
\]
Introducing the contact variable
\[
w:=Av+h(t,x),
\]
the coderivative of $G_t$ satisfies the following estimate. Specifically, for every
$p\in\mathbb R^n$,
\begin{equation}
\begin{aligned}
D^*G_t(x,v,u,\omega)(p)
\subset
\Bigg\{
\Big(
&
-\nabla_x f(t,x,u)^*p+\nabla_x h(t,x)^*\eta,
p+A^*\eta,
-\nabla_u f(t,x,u)^*p
\Big)\Big|\;
\\
&\qquad\qquad
\eta\in D^*N(C(t);\cdot)
\big(w,\omega+f(t,x,u)\big)(p)
\Bigg\}.
\end{aligned}
\label{4.14}
\end{equation}
If the moving set $C(t)$ is polyhedral, then every multiplier $\eta$ admits the representation
\begin{equation}
\eta
=
\sum_{j\in I_0(t,p)\cup I_>(t,p)}
\gamma_jx_*^j(t),
\label{4.15}
\end{equation}
where the index sets are taken from \eqref{4.index} and
\[
\gamma_j\in\mathbb R
\quad\text{for }j\in I_0(t,p),
\qquad
\gamma_j\ge0
\quad\text{for }j\in I_>(t,p).
\]
Moreover, when the active generating normals
$\{x_*^j(t)\}_{j\in I(t,w)}$
are linearly independent, the inclusion in
\eqref{4.14}
holds as an equality.
\end{theorem}

\begin{remark}\label{rem:coder}
The coderivative representation in Theorem~\ref{thm:coder} reflects the structure of the
implicit sweeping dynamics. Since the constraint depends on the velocity through the
contact point $w=Av+h(t,x)$, the coderivative of the implicit sweeping mapping naturally
splits into three components corresponding to the variables $(x,v,u)$. The term
involving $\nabla_xf$ and $\nabla_xh$ captures the sensitivity of the dynamics with respect
to the state, while the component $p+A^*\eta$ reflects the direct influence of the
velocity variable in the contact mapping. The last component $-\nabla_uf(t,x,u)^*p$
represents the effect of the control on the system dynamics. The multiplier $\eta$ arises
from the coderivative of the normal cone mapping to $C(t)$ and therefore represents the
contribution of the active constraints at the contact point $w$. In the polyhedral case,
representation \eqref{4.15} shows that this multiplier can be expressed as a linear
combination of the generating normals of $C(t)$.
\end{remark}

\begin{proof}
\ Fix $t\in[0,T]$ and define
\[
\xi:=(x,v,u), \qquad
w:=Av+h(t,x), \qquad
z:=\omega+f(t,x,u).
\]
Since
\[
\omega+f(t,x,u)\in N(C(t);w),
\]
it follows that
\[
z\in N(C(t);w),
\]
or equivalently,
\[
(w,z)\in\operatorname{gph}N(C(t);\cdot).
\]
Introduce the smooth mappings
\[
g(\xi):=v-f(t,x,u), \qquad
\phi(\xi):=Av+h(t,x),
\]
and define the set-valued mapping
\[
F(w):=N(C(t);w).
\]
Then the implicit sweeping mapping can be written as
\[
G_t(\xi)=g(\xi)+F(\phi(\xi)).
\]
Let $p\in\mathbb R^n$ and take any $q\in D^*G_t(x,v,u,\omega)(p)$. Since $g$ and $\phi$ are
continuously differentiable and the normal cone mapping $F$ has a closed graph
\cite{Mordukhovich06I}, the qualification conditions required for the coderivative sum and
chain rules are satisfied under the standing assumptions. Therefore, the coderivative sum
rule \cite[Theorem~1.62]{Mordukhovich06I} gives
\[
q\in\nabla g(\xi)^*p+D^*(F\circ\phi)(\xi,\omega-g(\xi))(p),
\]
provided that $\omega-g(\xi)\in F(\phi(\xi))$. Since $z=\omega+f(t,x,u)\in F(\phi(\xi))=N(C(t);w)$,
we use this element in the coderivative expression and rewrite the above as
\[
q\in\nabla g(\xi)^*p+D^*(F\circ\phi)(\xi,z)(p).
\]
The mapping $\phi$ is $\mathcal C^1$-smooth and $F$ is the normal cone mapping to the convex
set $C(t)$, so by the coderivative chain rule \cite[Theorem~1.66]{Mordukhovich06I} we get
\[
D^*(F\circ\phi)(\xi,z)(p)\subset\nabla\phi(\xi)^*\eta\quad\text{for some }\eta\in D^*F(w,z)(p).
\]
Consequently,
\[
q\in
\nabla g(\xi)^*p
+
\nabla\phi(\xi)^*\eta,
\]
for some
\[
\eta\in D^*F(w,z)(p)
=
D^*N(C(t);\cdot)(w,z)(p).
\]
Next, direct differentiation gives
\[
\frac{\partial g}{\partial x}=-\nabla_xf(t,x,u),\qquad
\frac{\partial g}{\partial v}=I,\qquad
\frac{\partial g}{\partial u}=-\nabla_uf(t,x,u),
\]
and hence
\[
\nabla g(\xi)^*p=
\begin{pmatrix}
-\nabla_xf(t,x,u)^*p\\
p\\
-\nabla_uf(t,x,u)^*p
\end{pmatrix},
\]
while
\[
\frac{\partial \phi}{\partial x}=\nabla_xh(t,x),\qquad
\frac{\partial \phi}{\partial v}=A,\qquad
\frac{\partial \phi}{\partial u}=0,
\qquad
\nabla\phi(\xi)^*\eta=
\begin{pmatrix}
\nabla_xh(t,x)^*\eta\\
A^*\eta\\
0
\end{pmatrix}.
\]
Combining these two expressions, we obtain
\[
q\in
\left\{
\begin{pmatrix}
-\nabla_xf(t,x,u)^*p+\nabla_xh(t,x)^*\eta\\
p+A^*\eta\\
-\nabla_uf(t,x,u)^*p
\end{pmatrix}
\,\middle|\,
\eta\in
D^*N(C(t);\cdot)(w,z)(p)
\right\},
\]
which is precisely the coderivative inclusion \eqref{4.14} with $z=\omega+f(t,x,u)$.
Finally, if $C(t)$ is polyhedral with the generating normals $\{x_*^j(t)\}$, the
coderivative representation for the normal cone mapping
\cite[Theorem~4.6]{HenrionMordukhovichNam10} implies that
\[
\eta=
\sum_{j\in I_0(t,p)\cup I_>(t,p)}
\gamma_jx_*^j(t),
\]
where
\[
\gamma_j\in\mathbb{R}
\quad\text{for }j\in I_0(t,p),
\qquad
\gamma_j\ge0
\quad\text{for }j\in I_>(t,p).
\]
Moreover, if the active generating normals
$\{x_*^j(t)\}_{j\in I(t,w)}$
are linearly independent, then the coderivative formula of \cite{HenrionMordukhovichNam10} is exact, and consequently
the inclusion \eqref{4.14} becomes an equality.
\end{proof}

The coderivative formula obtained above provides the key variational tool needed to
analyze the discrete-time optimal control problems associated with the implicit sweeping
process. In particular, it allows us to compute the normal cone to the graph of the
constraint system governing the discrete inclusions, which leads to the derivation of
necessary optimality conditions via generalized differentiation techniques. In the next
section, this coderivative representation is incorporated into the discrete optimality
system and subsequently used to pass to the continuous-time limit.

\section{Necessary Optimality Conditions for the Discrete Implicit Sweeping Problem}\label{sec:discreteNOC}

This section establishes first-order necessary optimality conditions for the discrete implicit sweeping problem introduced in Section~\ref{sec:discrete}. The analysis follows the variational approach developed in Section~\ref{sec:coder}: by combining the finite-dimensional multiplier rule with the coderivative formula of Theorem~\ref{thm:coder}, we obtain a discrete Euler--Lagrange system that will later serve as
the basis for passing to the continuous-time limit.

\begin{theorem}\label{thm:discreteEL}
Let $m \in \mathbb{N}$, and
\[
\Delta_m=\{0=t_m^0<t_m^1<\cdots<t_m^{2^m}=T\}
\]
be the uniform partition of $[0,T]$ with mesh size
\[
h_m:= t^{i+1}_m - t^i_m =\frac{T}{2^m},
\]
and let
\[
(\bar{x}_m,\bar{u}_m)
=
(\bar{x}_m^0,\ldots,\bar{x}_m^{2^m},
\bar{u}_m^0,\ldots,\bar{u}_m^{2^m-1})
\]
be an optimal solution of the localized discrete implicit sweeping problem $(P_m)$
generated by
\begin{equation}
\frac{x_m^{i+1}-x_m^i}{h_m}
\in
f(t_m^i,x_m^i,u_m^i)
-
N\!\left(
C(t_m^i);
A\frac{x_m^{i+1}-x_m^i}{h_m}
+h(t_m^i,x_m^i)
\right),
\label{5.1}
\end{equation}
for $i=0,\ldots,2^m-1$.
Define the discrete velocities
\[
\bar v_m^i
:=
\frac{\bar x_m^{i+1}-\bar x_m^i}{h_m},
\]
together with the associated velocity contact point
\[
\bar w_m^i
:=
A\bar v_m^i+h(t_m^i,\bar x_m^i),
\qquad
i=0,\ldots,2^m-1,
\]
so that $\bar w_m^i\in C(t_m^i)$ for all $i$, which is automatically ensured by \eqref{5.1}.
For each $t\in[0,T]$, define the graph of the implicit sweeping constraint by
\[
\operatorname{gph}F(t)
:=
\Big\{
(x,u,v)\in\mathbb{R}^n\times\mathbb{R}^d\times\mathbb{R}^n
\ \Big|\
0\in
v-f(t,x,u)
+
N(C(t);Av+h(t,x))
\Big\}.
\]

Assume that hypotheses {\rm(H1)--(H6)} are satisfied and that the terminal cost
$\varphi$ is locally Lipschitz in a neighborhood of
$\bar x_m^{2^m}$.
Then there exist
\[
\lambda_m\ge0,\qquad
p_m^i\in\mathbb{R}^n,\qquad
\psi_m^i\in N(\bar u_m^i;U),
\]
together with terminal multipliers (terminal polyhedral multipliers, when present)
\[
\xi_m=(\xi_m^1,\ldots,\xi_m^s)\in\mathbb{R}_+^s,
\]
such that the following assertions hold.

\begin{enumerate}

\item[\rm(a)] \textbf{Nontrivial multiplier condition.}
The multipliers are not all zero, namely
\[
\lambda_m
+\|\xi_m\|
+\sum_{i=0}^{2^m}\|p_m^i\|
+\sum_{i=0}^{2^m-1}\|\psi_m^i\|
\neq0.
\]

\item[\rm(b)] \textbf{Complementary slackness.}
If the terminal constraints are given by
\[
\langle x_*^j,\bar x_m^{2^m}\rangle\le c_j(T),
\qquad
j=1,\ldots,s,
\]
then
\[
\xi_m^j
\Big(
\langle x_*^j,\bar x_m^{2^m}\rangle
-c_j(T)
\Big)
=0.
\]

\item[\rm(c)] \textbf{Terminal transversality relation.}
There exists
\[
\vartheta_m^{2^m}\in\partial\varphi(\bar x_m^{2^m})
\]
such that
\[
-p_m^{2^m}
=
\lambda_m\vartheta_m^{2^m}
+
\sum_{j=1}^{s}\xi_m^j x_*^j.
\]

\item[\rm(d)] \textbf{Discrete Euler--Lagrange inclusion.}
For every
$i=0,\ldots,2^m-1$,
the adjoint variables satisfy
\[
\left(
\frac{p_m^{i+1}-p_m^i}{h_m},
-\frac{1}{h_m}\psi_m^i,
p_m^{i+1}
\right)
\in
N\!\left(
(\bar x_m^i,\bar u_m^i,\bar v_m^i);
\operatorname{gph}F(t_m^i)
\right).
\]
\end{enumerate}
The conditions {\rm(a)--(d)} constitute the discrete Euler--Lagrange system associated
with the implicit sweeping dynamics and provide necessary optimality conditions for the
localized discrete problem $(P_m)$.
\end{theorem}

\begin{proof}
\ Fix $m\in\mathbb N$ and write $h_m=T/2^m$ and $t_m^i=ih_m$ for
$i=0,\ldots,2^m$. We first rewrite the discrete inclusion as a finite-dimensional
optimization problem. Introduce independent velocity variables
$v_m^i\in\mathbb R^n$ and impose the Euler relations
\[
g_i(z):=x_m^{i+1}-x_m^i-h_mv_m^i=0,
\qquad i=0,\ldots,2^m-1.
\]
Then $v_m^i=(x_m^{i+1}-x_m^i)/h_m$. Define
\[
w_m^i:=Av_m^i+h(t_m^i,x_m^i),
\qquad i=0,\ldots,2^m-1.
\]
Thus, the discrete sweeping inclusion becomes
\[
v_m^i\in f(t_m^i,x_m^i,u_m^i)-N(C(t_m^i);w_m^i),
\qquad i=0,\ldots,2^m-1.
\]

Collect all decision variables in the vector
\[
z:=(x_m^0,\ldots,x_m^{2^m},u_m^0,\ldots,u_m^{2^m-1},
v_m^0,\ldots,v_m^{2^m-1})\in\mathbb R^{(2^m+1)n+2^md+2^mn}.
\]
The finite-dimensional problem consists of minimizing
\[
\varphi(x_m^{2^m})
\]
subject to the Euler equalities $g_i(z)=0$, the control constraints
$u_m^i\in U$, the terminal constraints (when present)
\[
h_j(z):=\langle x_*^j,x_m^{2^m}\rangle-c_j(T)\le0,
\qquad j=1,\ldots,s,
\]
and the graph constraints generated by the discrete sweeping dynamics. Here $\varphi$ is
locally Lipschitz around the reference terminal state $\bar x_m^{2^m}$. This reformulation
transforms the discrete implicit sweeping problem into a finite-dimensional mathematical
program with equality constraints, inequality constraints, control constraints, and
geometric constraints involving normal cones, which allows us to apply generalized
multiplier rules.

For each fixed $t\in[0,T]$, define
\[
F(t):=
\Big\{
(x,u,v)\in\mathbb R^n\times\mathbb R^d\times\mathbb R^n
\;\Big|\;
v\in f(t,x,u)-N(C(t);Av+h(t,x))
\Big\}.
\]
Then the sweeping inclusion at the grid point $t_m^i$ is equivalently written as
the graph constraint
\[
(x_m^i,u_m^i,v_m^i)\in\operatorname{gph}F(t_m^i),
\qquad i=0,\ldots,2^m-1.
\]
This representation allows the implicit sweeping dynamics to be treated as a geometric
constraint in a finite-dimensional mathematical program, enabling the application of
normal cone and coderivative calculus.

Since $\bar z$ is a local minimizer of this finite-dimensional problem, the
generalized multiplier rule \cite[Theorem~5.24]{Mordukhovich06II} applies; the standing
assumptions (H1)--(H6) ensure the required constraint qualification for the geometric
constraints involving the normal cone mapping. Hence there
exist multipliers
\[
\lambda_m\ge0,\qquad
\xi_m=(\xi_m^1,\ldots,\xi_m^s)\in\mathbb R_+^s,
\qquad
p_m^{i+1}\in\mathbb R^n,
\]
associated with the cost, the terminal inequality constraints, and the Euler equalities,
respectively, together with control multipliers
\[
\psi_m^i\in N(\bar u_m^i;U),
\qquad i=0,\ldots,2^m-1,
\]
and graph-normal multipliers
\[
(\alpha_m^i,\beta_m^i,\gamma_m^i)
\in
N\big((\bar x_m^i,\bar u_m^i,\bar v_m^i);
\operatorname{gph}F(t_m^i)\big)
\]
for $i=0,\ldots,2^m-1$, not all zero.

Let
$\Phi(z):=\varphi(x_m^{2^m})$
denote the objective function. The multiplier rule gives the stationarity
inclusion
\begin{equation}
0\in
\lambda_m\partial\Phi(\bar z)
+
\sum_{j=1}^s\xi_m^j\nabla h_j(\bar z)
+
\sum_{i=0}^{2^m-1}\nabla g_i(\bar z)^*p_m^{i+1}
+
\sum_{i=0}^{2^m-1}\mathcal N_U^i
+
\sum_{i=0}^{2^m-1}(\alpha_m^i,\beta_m^i,\gamma_m^i),
\label{5.stat}
\end{equation}
where $\partial\Phi(\bar z)$ acts only on the $x_m^{2^m}$-component and is zero on all
other components,
\[
\mathcal N_U^i:=\big\{(0,\ldots,0,\eta,0,\ldots,0)\;\big|\;\eta\in N(\bar u_m^i;U)\big\}
\]
denotes the contribution of the normal cone $N(\bar u_m^i;U)$ embedded in the
$u_m^i$-block, and each $(\alpha_m^i,\beta_m^i,\gamma_m^i)$ is understood as embedded into the
full space of $z$ at the $(x_m^i,u_m^i,v_m^i)$-block. The stationarity inclusion
\eqref{5.stat} is componentwise with respect to the block decomposition of $z$, and we
analyze it separately in the blocks corresponding to $x_m^i,u_m^i,v_m^i$.

Since $\Phi$ depends only
on the terminal variable $x_m^{2^m}$, its subgradient is zero in every block except the
$x_m^{2^m}$-block, where it equals $\partial\varphi(\bar x_m^{2^m})$. Thus there exists
\[
\vartheta_m^{2^m}\in\partial\varphi(\bar x_m^{2^m})
\]
such that the terminal block of $\lambda_m\partial\Phi(\bar z)$ is
$\lambda_m\vartheta_m^{2^m}$. Also,
\[
\nabla_{x_m^{2^m}}h_j(\bar z)=x_*^j,
\qquad j=1,\ldots,s.
\]
Among the Euler equalities, only the last one, $g_{2^m-1}(z)=x_m^{2^m}-x_m^{2^m-1}-h_mv_m^{2^m-1}$,
contributes to the
$x_m^{2^m}$-block; since $\nabla_{x_m^{2^m}}g_{2^m-1}=I_n$, its contribution is $p_m^{2^m}$. Therefore the terminal
block of the stationarity inclusion gives
\[
0\in
\lambda_m\vartheta_m^{2^m}
+
\sum_{j=1}^s\xi_m^jx_*^j
+
p_m^{2^m},
\]
or equivalently,
\[
-p_m^{2^m}
=
\lambda_m\vartheta_m^{2^m}
+
\sum_{j=1}^s\xi_m^jx_*^j.
\]
This proves the terminal transversality condition.

The complementary slackness condition follows from the multiplier rule for the
terminal inequalities. Indeed, at the optimal solution $\bar z$ feasibility gives
$h_j(\bar z)\le0$, and the inequality constraints produce multipliers $\xi_m^j\ge0$ with
$\xi_m^jh_j(\bar z)=0$, $j=1,\ldots,s$. Thus, if the constraint is inactive ($h_j(\bar z)<0$),
then $\xi_m^j=0$, and if $\xi_m^j>0$, then necessarily $h_j(\bar z)=0$. Since
\[
h_j(\bar z)=
\langle x_*^j,\bar x_m^{2^m}\rangle-c_j(T),
\]
we obtain
\[
\xi_m^j
\Big(
\langle x_*^j,\bar x_m^{2^m}\rangle-c_j(T)
\Big)
=0,
\qquad j=1,\ldots,s.
\]

It remains to derive the graph-normal Euler--Lagrange inclusion. Fix
$i=0,\ldots,2^m-1$ and examine the block corresponding to
$(x_m^i,u_m^i,v_m^i)$. The objective and terminal constraints do not contribute
to this block unless $i=2^m$, which is not the case here. Moreover, $\mathcal N_U^k$
contributes only to the $u_m^k$-block, and $(\alpha_m^k,\beta_m^k,\gamma_m^k)$ contributes only
to the $(x_m^k,u_m^k,v_m^k)$-block. Since $g_k$ does not depend on $u$, the Euler
equality terms involving $x_m^i$ and $v_m^i$ give the contribution
\[
(p_m^i-p_m^{i+1},\,0,\,-h_mp_m^{i+1}).
\]
The control constraint contributes a vector of the form
\[
(0,\psi_m^i,0),
\qquad
\psi_m^i\in N(\bar u_m^i;U).
\]
Together with the graph-normal multiplier, the stationarity condition gives
\[
(p_m^{i+1}-p_m^i,\,-\psi_m^i,\,h_mp_m^{i+1})
\in
N\big((\bar x_m^i,\bar u_m^i,\bar v_m^i);
\operatorname{gph}F(t_m^i)\big).
\]
Dividing by $h_m>0$ and using the conic property of the normal cone yields
\[
\left(
\frac{p_m^{i+1}-p_m^i}{h_m},
-\frac1{h_m}\psi_m^i,
p_m^{i+1}
\right)
\in
N\big((\bar x_m^i,\bar u_m^i,\bar v_m^i);
\operatorname{gph}F(t_m^i)\big),
\]
which is the desired Euler--Lagrange graph-normal inclusion.

Finally, the nontriviality condition is the nonzero multiplier condition supplied
by the generalized multiplier rule, rewritten in terms of the reduced
multipliers $\lambda_m,\xi_m,p_m^i,$ and $\psi_m^i$. Thus
\[
\lambda_m+\|\xi_m\|
+\sum_{i=0}^{2^m}\|p_m^i\|
+\sum_{i=0}^{2^m-1}\|\psi_m^i\|
\neq0.
\]
The proof is complete.
\end{proof}

Theorem~\ref{thm:discreteEL} provides a geometric necessary optimality condition
expressed in terms of normals to the graph of the implicit sweeping mapping. In order to
obtain explicit adjoint difference relations, we now characterize this graph-normal
inclusion via the coderivative formula derived in Section~\ref{sec:coder}.

\begin{theorem}\label{thm:discreteadj}
Assume that the hypotheses of Theorem~\ref{thm:discreteEL} are satisfied. Let
$(\bar{x}_m,\bar{u}_m)$ be an optimal solution of the localized discrete
implicit sweeping problem $(P_m)$. For each
$i=0,\ldots,2^{m}-1$, define the discrete velocity and the corresponding
velocity contact point by
\begin{equation*}
\bar{v}_m^{\,i}:=
\frac{\bar{x}_m^{\,i+1}-\bar{x}_m^{\,i}}{h_m},
\qquad
\bar{w}_m^{\,i}:=
A\bar{v}_m^{\,i}+h(t_m^i,\bar{x}_m^{\,i})
\in C(t_m^i).
\end{equation*}

Then there exist multipliers
\[
\lambda_m\ge0,\qquad
p_m^i\in\mathbb{R}^n,\; i=0,\ldots,2^m,
\]
\[
\psi_m^i\in N(\bar{u}_m^{\,i};U),\qquad
\eta_m^i\in\mathbb{R}^n,\; i=0,\ldots,2^m-1,
\]
together with terminal multipliers
\[
\xi_m=(\xi_m^1,\ldots,\xi_m^s)\in\mathbb{R}_+^s,
\]
such that the following assertions hold.

\noindent\textbf{(a) Nontriviality.}

\begin{equation*}
\lambda_m
+\|\xi_m\|
+\sum_{i=0}^{2^m}\|p_m^i\|
+\sum_{i=0}^{2^m-1}\|\psi_m^i\|
\neq0.
\end{equation*}

\noindent\textbf{(b) Terminal transversality condition.}
There exists
\[
\vartheta_m^{\,2^m}\in
\partial\varphi(\bar{x}_m^{\,2^m})
\]
such that
\begin{equation*}
-p_m^{\,2^m}
=
\lambda_m\vartheta_m^{\,2^m}
+
\sum_{j=1}^{s}\xi_m^j x_*^j.
\end{equation*}

\noindent\textbf{(c) Discrete adjoint equation.}
For every
$i=0,\ldots,2^m-1$,
\begin{equation*}
\frac{p_m^{\,i+1}-p_m^{\,i}}{h_m}
=
-\nabla_xf(t_m^i,\bar{x}_m^{\,i},\bar{u}_m^{\,i})^*
\bigl(
p_m^{\,i+1}-A^*\eta_m^{\,i}
\bigr)
+
\nabla_xh(t_m^i,\bar{x}_m^{\,i})^*
\eta_m^{\,i}.
\end{equation*}

\noindent\textbf{(d) Control stationarity.}
For every
$i=0,\ldots,2^m-1$,
\begin{equation*}
-\frac{1}{h_m}\psi_m^{\,i}
=
-\nabla_uf(t_m^i,\bar{x}_m^{\,i},\bar{u}_m^{\,i})^*
\bigl(
p_m^{\,i+1}-A^*\eta_m^{\,i}
\bigr).
\end{equation*}

\medskip

\noindent\textbf{(e) Representation of the normal-cone multiplier.}
For each
$i=0,\ldots,2^m-1$,
\begin{equation*}
\eta_m^{\,i}
=
\sum_{j\in I(\bar{w}_m^{\,i})\cup J(\bar{w}_m^{\,i})}
\gamma_m^{ij}\,
x_*^j(t_m^i),
\end{equation*}
where the coefficients satisfy
\[
\gamma_m^{ij}\ge0,
\qquad
j\in I(\bar{w}_m^{\,i}),
\]
and
\[
\gamma_m^{ij}\in\mathbb{R},
\qquad
j\in J(\bar{w}_m^{\,i}).
\]
Here, in accordance with Theorem~\ref{thm:coder} and \eqref{4.index}, the index sets
$I(\bar w_m^i)$ and $J(\bar w_m^i)$ stand for the sets $I_>$ and $I_0$, respectively,
evaluated at the contact point $\bar w_m^i$ in the direction $r_m^i:=p_m^{i+1}-A^*\eta_m^i$.
\end{theorem}

\medskip

\begin{proof}
\ Fix an arbitrary index $i\in\{0,\ldots,2^m-1\}$. From the discrete
Euler--Lagrange inclusion established in Theorem~\ref{thm:discreteEL}(d), we have
\[
\left(
\frac{p_m^{i+1}-p_m^i}{h_m},
-\frac{1}{h_m}\psi_m^i,
p_m^{i+1}
\right)
\in
N\!\left(
(\bar x_m^i,\bar u_m^i,\bar v_m^i);
\operatorname{gph}F(t_m^i)
\right).
\]
For convenience, introduce the graph-normal triple
\[
(q_x^i,q_u^i,q_v^i):=
\left(
\frac{p_m^{i+1}-p_m^i}{h_m},
-\frac1{h_m}\psi_m^i,
p_m^{i+1}
\right).
\]
Then the above inclusion is equivalent to
\[
(q_x^i,q_u^i,q_v^i)
\in
N\!\left(
(\bar x_m^i,\bar u_m^i,\bar v_m^i);
\operatorname{gph}F(t_m^i)
\right).
\]

By the definition \eqref{4.coder} of the coderivative of the multifunction $F(t):(x,u)\mapsto v$,
\[
D^*F(t)(\bar x,\bar u,\bar v)(v^*)
:=
\left\{
(x^*,u^*)
\,\middle|\,
(x^*,u^*,-v^*)
\in
N((\bar x,\bar u,\bar v);\operatorname{gph}F(t))
\right\},
\]
and choosing $v_m^{*i}:=-q_v^i=-p_m^{i+1}$, we obtain
\[
(q_x^i,q_u^i)
\in
D^*F(t_m^i)
(\bar x_m^i,\bar u_m^i,\bar v_m^i)
(-p_m^{i+1}).
\]

To exploit the coderivative formula established in Section~\ref{sec:coder}, we express
the sweeping dynamics in implicit form. Define
\[
G_t(x,v,u)
:=
v-f(t,x,u)+N(C(t);Av+h(t,x)).
\]
Then
\[
v\in F(t)(x,u)
\quad\Longleftrightarrow\quad
0\in G_t(x,v,u).
\]

Since
$(\bar x_m^i,\bar u_m^i,\bar v_m^i)\in\operatorname{gph}F(t_m^i)$,
it follows that
\[
0\in
G_{t_m^i}
(\bar x_m^i,\bar v_m^i,\bar u_m^i).
\]
Therefore, by the coderivative characterization derived in Section~\ref{sec:coder}, there
exists a vector $r_m^i\in\mathbb R^n$ satisfying
\[
(q_x^i,q_v^i,q_u^i)
\in
D^*G_{t_m^i}
(\bar x_m^i,\bar v_m^i,\bar u_m^i,0)
(r_m^i),
\]
where we have reordered the triple from $(x,u,v)$ to $(x,v,u)$ in order to match the
variable order in $G_t$.

Next, decompose the implicit mapping as
\[
G_t(x,v,u)=g_t(x,v,u)+H_t(x,v),
\]
where
\[
g_t(x,v,u):=v-f(t,x,u),
\qquad
H_t(x,v)
:=
N(C(t);Av+h(t,x)).
\]

Applying the coderivative sum rule yields
\[
D^*G_t(\bar z,0)(r)
\subset
\nabla g_t(\bar z)^*r
+
D^*H_t(\bar x,\bar v)(r),
\]
where $\bar z=(\bar x,\bar v,\bar u)$.

The smooth component contributes
\[
\nabla g_t(\bar z)^*r
=
\left(
-\nabla_xf(t,\bar x,\bar u)^*r,\;
r,\;
-\nabla_uf(t,\bar x,\bar u)^*r
\right).
\]

For the nonsmooth component, observe that
\[
H_t(x,v)
=
N(C(t);Av+h(t,x))
\]
is the composition of the smooth mapping $\Phi(x,v):=Av+h(t,x)$ with the normal
cone mapping. Since $\nabla_x\Phi=\nabla_xh(t,x)$, $\nabla_v\Phi=A$, and $\nabla_u\Phi=0$,
Theorem~\ref{thm:coder} provides the existence of a vector
$\eta_m^i\in\mathbb R^n$ such that the coderivative contribution of $H_t$ has the form
\[
\nabla\Phi(\bar x,\bar v)^*\eta_m^i
=
\left(
\nabla_xh(t,\bar x)^*\eta_m^i,
A^*\eta_m^i,
0
\right),
\]
where
\[
\eta_m^i
\in
D^*N(C(t_m^i);\cdot)
(\bar w_m^i,\bar\xi_m^i)(r_m^i),
\qquad
\bar\xi_m^i
:=
f(t_m^i,\bar x_m^i,\bar u_m^i)-\bar v_m^i.
\]

Combining the contributions from $g_t$ and $H_t$, we conclude that
\[
(q_x^i,q_v^i,q_u^i)
=
\Big(
-\nabla_xf(t_m^i,\bar x_m^i,\bar u_m^i)^*r_m^i
+\nabla_xh(t_m^i,\bar x_m^i)^*\eta_m^i,\;
r_m^i+A^*\eta_m^i,\;
-\nabla_uf(t_m^i,\bar x_m^i,\bar u_m^i)^*r_m^i
\Big).
\]

Using the definition of the graph-normal triple, we have
\[
q_x^i=\frac{p_m^{i+1}-p_m^i}{h_m},
\qquad
q_u^i=-\frac1{h_m}\psi_m^i,
\qquad
q_v^i=p_m^{i+1}.
\]

Comparing the second component of the previous identity gives
\[
p_m^{i+1}
=
r_m^i+A^*\eta_m^i,
\]
or equivalently,
\begin{equation}
r_m^i
=
p_m^{i+1}-A^*\eta_m^i.
\label{5.r}
\end{equation}

Substituting this expression for $r_m^i$ into the first component yields
\[
\frac{p_m^{i+1}-p_m^i}{h_m}
=
-\nabla_xf(t_m^i,\bar x_m^i,\bar u_m^i)^*
\bigl(p_m^{i+1}-A^*\eta_m^i\bigr)
+\nabla_xh(t_m^i,\bar x_m^i)^*\eta_m^i,
\]
which is precisely the adjoint difference equation stated in
Theorem~\ref{thm:discreteadj}(c).

Likewise, the third component gives
\[
-\frac1{h_m}\psi_m^i
=
-\nabla_uf(t_m^i,\bar x_m^i,\bar u_m^i)^*r_m^i.
\]
Replacing $r_m^i$ by
$p_m^{i+1}-A^*\eta_m^i$
produces
\[
-\frac1{h_m}\psi_m^i
=
-\nabla_uf(t_m^i,\bar x_m^i,\bar u_m^i)^*
\bigl(p_m^{i+1}-A^*\eta_m^i\bigr),
\]
which establishes the control stationarity condition in
Theorem~\ref{thm:discreteadj}(d).

It remains to characterize the multiplier $\eta_m^i$.
Since every moving set $C(t)$ is assumed to be polyhedral with
generating normals
$\{x_*^1(t),\ldots,x_*^s(t)\}$,
Theorem~\ref{thm:coder} provides an explicit description of the coderivative of the
normal cone mapping. Consequently,
\[
\eta_m^i
=
\sum_{j\in I(\bar w_m^i)\cup J(\bar w_m^i)}
\gamma_m^{ij}x_*^j(t_m^i),
\]
where the coefficients satisfy
\[
\gamma_m^{ij}\ge0,
\qquad
j\in I(\bar w_m^i),
\]
while
\[
\gamma_m^{ij}\in\mathbb R,
\qquad
j\in J(\bar w_m^i).
\]
This is exactly the normal-cone representation stated in
Theorem~\ref{thm:discreteadj}(e).

Since the index
$i\in\{0,\ldots,2^m-1\}$
was arbitrary, the adjoint equation, the control stationarity relation,
and the representation of the normal-cone multiplier hold at every
discrete time step. Together with the nontriviality and terminal
conditions established in Theorem~\ref{thm:discreteEL}, all assertions of
Theorem~\ref{thm:discreteadj} follow. Therefore, the proof is complete.
\end{proof}

\section{Continuous-Time Necessary Optimality Conditions}\label{sec:continuous}

This section derives first-order necessary optimality conditions for the continuous-time implicit sweeping control problem. The results are obtained by passing to the limit in the discrete optimality system established in Section~\ref{sec:discreteNOC} and are formulated in terms of adjoint variables and measure-valued multipliers. The convergence analysis relies on the discrete Euler--Lagrange conditions together with compactness properties of the adjoint variables, normal-cone multipliers, and measure-valued dual variables. The resulting limiting system yields a continuous-time maximum principle for the implicit sweeping process.

\begin{definition}\label{def:continuous_solution}
Let
\[
C:=
\left\{
w\in\mathbb{R}^{n}
\;\middle|\;
\langle x_*^{\,j},w\rangle\le c_j,\;
j=1,\ldots,s
\right\}
\]
be the polyhedral moving set.
A trajectory-control pair
\[
(x(\cdot),u(\cdot))
\in
W^{1,2}([0,T];\mathbb{R}^{n})
\times
L^{2}([0,T];\mathbb{R}^{d})
\]
is called an admissible trajectory (solution) of the controlled implicit sweeping system provided that
\[
\dot{x}(t)\in
f(t,x(t),u(t))
-
N(C;w(t))
\]
for almost every $t\in[0,T]$, where $w(t):=A\dot{x}(t)+h(t,x(t))$ denotes the associated velocity-contact variable.

Since the normal cone to a polyhedron admits a conical representation, the differential inclusion is equivalently expressed through measurable multipliers
\[
\eta^j(\cdot)\in
L^2([0,T];\mathbb{R}_+),
\qquad
j=1,\ldots,s,
\]
satisfying
\[
\dot{x}(t)
=
f(t,x(t),u(t))
-
\sum_{j=1}^{s}
\eta^j(t)x_*^{\,j}\
\text{ for almost every $t\in[0,T]$.}\]

Introducing the active-index mapping
\[
I(w(t))
:=
\left\{
j\in\{1,\ldots,s\}
\;\middle|\;
\langle x_*^{\,j},w(t)\rangle=c_j
\right\},
\]
the multipliers satisfy the complementarity condition
$\eta^j(t)=0$ whenever $j\notin I(w(t)).$

The normal cone is said to be active along $x(\cdot)$ at time $t$ whenever $I(w(t))\neq\emptyset,$ and we denote the corresponding active-time set by
\[
E_0:=
\left\{
t\in[0,T]
\mid
I(w(t))\neq\emptyset
\right\}.
\]
We say that the normal cone is active along the trajectory if $E_0=[0,T]$.
\end{definition}

\begin{remark}\label{rem:implicit-vs-explicit}
The implicit sweeping dynamics differ from explicit sweeping processes (see, e.g.,
\cite{CMN20,CHHM16}) in that the normal cone is evaluated at the velocity-contact point
$w(t)=A\dot x(t)+h(t,x(t))$ rather than directly at the state variable. In other words,
the constraint set acts on the transformed variable $w(t)$ instead of $x(t)$. This
introduces an additional level of coupling between the state and the velocity in the
dynamics. As a result, the analysis of optimality conditions requires working with the
auxiliary mapping $w(\cdot)$ and the associated active constraint set $I(w(t))$.
\end{remark}

\begin{theorem}\label{thm:continuous_mp}
Let $(\bar x(\cdot),\bar u(\cdot))$ be a relaxed local minimizer of problem $(P)$
belonging to
\[
W^{1,2}([0,T];\mathbb R^n)\times L^2([0,T];\mathbb R^d).
\]

Assume that hypotheses {\rm(H1)--(H6)} are satisfied, that the reference control
$\bar u(\cdot)$ is of bounded variation, and that the Linear Independence
Constraint Qualification (LICQ) holds along the velocity-contact trajectory
\[
\bar w(t):=A\dot{\bar x}(t)+h(t,\bar x(t)).
\]

Suppose, moreover, that there exists $\bar t\in(0,T)$ such that
\[
\bar w(t)\in\operatorname{int}C
\quad\text{for }t\in[0,\bar t),
\qquad
\bar w(t)\in\operatorname{bd}C
\quad\text{for }t\in[\bar t,T],
\]
and that the active index set $I(\bar w(t))$ is unchanged on the interval
$[\bar t,T]$.

Under these assumptions, there exist a scalar multiplier
$\mu\ge0$, an adjoint arc
$p(\cdot)\in W^{1,2}([0,T];\mathbb R^n)$,
a vector-valued measure
$\gamma=(\gamma^1,\ldots,\gamma^s)\in C^*([0,T];\mathbb R^s)$,
and multiplier functions
\[
\eta^j(\cdot)\in L^2([0,T];\mathbb R_+),
\qquad
j=1,\ldots,s,
\]
for which the following conditions are valid.

\begin{enumerate}
\item[(i)] \textbf{Primal-dual dynamics.}

\begin{enumerate}

\item[(a)]
The optimal trajectory admits the representation
\begin{equation}
  \dot{\bar x}(t)
=
f(t,\bar x(t),\bar u(t))
-
\sum_{j=1}^{s}\eta^j(t)x_*^j
\quad
\text{for a.e. }t\in[0,T],  \label{6.1}
\end{equation}

where
\[
\eta^j(t)=0
\quad
\text{for a.e. }t
\text{ whenever }
j\notin I(\bar w(t)).
\]

\item[(b)]
The adjoint variable satisfies
\begin{equation}
    \dot p(t)
=
-\nabla_xf(t,\bar x(t),\bar u(t))^*q(t)
-
\nabla_xh(t,\bar x(t))^*A^*q(t)
\quad
\text{for a.e. }t\in[0,T],\label{6.2}
\end{equation}
where
\begin{equation}
   q(t)
=
p(t)
-
\int_{(t,T]}
\sum_{j=1}^{s}
x_*^j\,d\gamma^j(\tau),
\qquad
t\in[0,T], \label{6.3}
\end{equation}
and
\[
p(T)=q(T).
\]

\end{enumerate}

\item[(ii)] \textbf{Maximum principle.}

Introducing the Hamiltonian-type function
\begin{equation}
\psi(t,u)
:=
\langle\nabla_uf(t,\bar x(t),u),q(t)\rangle,\label{6.4}
\end{equation}
the reference control satisfies
\begin{equation}
    \psi(t,\bar u(t))
=
\max_{u\in U}
\psi(t,u)
\quad
\text{for a.e. }t\in[0,T].\label{6.5}
\end{equation}

\item[(iii)] \textbf{Dynamic complementarity.}

For almost every $t\in[0,T]$ and every
$j=1,\ldots,s$,
\begin{equation}
    \langle x_*^j,\bar w(t)\rangle<c_j
\Longrightarrow
\eta^j(t)=0
\quad\text{and}\quad
\langle x_*^j,q(t)\rangle=0.\label{6.6}
\end{equation}

\item[(iv)] \textbf{Terminal conditions.}

There exist terminal multipliers
$\eta^j(T)\ge0$,
$j\in I(\bar w(T))$,
such that
\begin{equation}
    -p(T)
-
\sum_{j\in I(\bar w(T))}
\eta^j(T)x_*^j
\in
\mu\,\partial\varphi(\bar x(T)),
\label{6.7}
\end{equation}
and
\begin{equation}
    \sum_{j\in I(\bar w(T))}
\eta^j(T)x_*^j
\in
N(\bar w(T);C).\label{6.8}
\end{equation}

\item[(v)] \textbf{Measure nonatomicity.}

Whenever $t\in[0,T)$ satisfies
\begin{equation}
    \langle x_*^j,\bar w(t)\rangle<c_j,
\qquad
j=1,\ldots,s,\label{6.9}
\end{equation}
there exists a neighborhood $V_t$ of $t$
such that
\[
\gamma(V_t)=0.
\]

\item[(vi)] \textbf{Nontriviality.}

The multipliers satisfy
\begin{equation}
    (\mu,p,\gamma)\neq(0,0,0),\label{6.10}
\end{equation}
together with the support condition
\begin{equation}
    \operatorname{supp}(\gamma)\subset E_0.\label{6.11}
\end{equation}
In addition, if
\begin{equation}
    \langle x_*^j,\bar w(t)\rangle<c_j,
\qquad
\forall\,t\in[0,T],\;
j=1,\ldots,s,\label{6.12}
\end{equation}
then the enhanced nontriviality condition holds and one may choose the normalization
\[
\mu=1.
\]

\end{enumerate}
\end{theorem}

\begin{proof}
\ Let $(\bar{x}(\cdot),\bar{u}(\cdot))$ be the relaxed local minimizer under
consideration. By Proposition~\ref{prop:exist} and Theorem~\ref{thm:optconv}, for every sufficiently large $m$ there exists
an optimal solution $(\bar{x}_m(\cdot),\bar{u}_m(\cdot))$ of the localized
discrete problem $(P_m)$ such that
\[
\bar{x}_m\to\bar{x}\quad\text{strongly in }W^{1,2}([0,T];\mathbb{R}^n),
\qquad
\bar{u}_m\to\bar{u}\quad\text{strongly in }L^2([0,T];\mathbb{R}^d).
\]
We derive each of the claimed conditions by passing to the limit in the discrete
necessary conditions obtained in Theorems~\ref{thm:discreteEL} and~\ref{thm:discreteadj}.

\medskip
\noindent
\textbf{Step 1: Passage to the limit in the primal dynamics.}
Let
\[
\Delta_m:=\{0=t_m^0<t_m^1<\cdots<t_m^{2^m}=T\},
\qquad
h_m:=t_m^{i+1}-t_m^i=\frac{T}{2^m},
\]
and define, on every interval $[t_m^i,t_m^{i+1})$,
\[
\theta_m(t):=\frac{\theta_m^i}{h_m},
\qquad
\theta_m^i:=
\int_{t_m^i}^{t_m^{i+1}}
\left(
\dot{\bar{x}}(t)
-
\frac{\bar{x}(t_m^{i+1})-\bar{x}(t_m^i)}{h_m}
\right)dt.
\]
Similarly, let
\[
\varphi_m(t):=\frac{\varphi_m^i}{h_m},
\qquad
\varphi_m^i:=
\int_{t_m^i}^{t_m^{i+1}}
\big(\bar{u}(t)-\bar{u}(t_m^i)\big)\,dt.
\]
Jensen's inequality gives, similarly to \cite{CMN20},
\[
\begin{aligned}
\int_0^T\|\theta_m(t)\|^2dt
&=
\sum_{i=0}^{2^m-1}\frac{\|\theta_m^i\|^2}{h_m}                                                     \\
&\leq
\sum_{i=0}^{2^m-1}
\int_{t_m^i}^{t_m^{i+1}}
\left\|
\dot{\bar{x}}(t)
-
\frac{\bar{x}(t_m^{i+1})-\bar{x}(t_m^i)}{h_m}
\right\|^2dt
=
\int_0^T\|\dot{\bar{x}}(t)-\dot{\bar{x}}_m(t)\|^2dt,
\end{aligned}
\]
where $\dot{\bar{x}}_m$ denotes the piecewise constant derivative of the
piecewise linear interpolant of $\bar x$. Hence, by Theorem~\ref{thm:approx},
\begin{equation}
\int_0^T\|\theta_m(t)\|^2dt\longrightarrow 0.
\label{eq:theta-convergence}
\end{equation}
In the same way,
\[
\int_0^T\|\varphi_m(t)\|^2dt
\leq
\int_0^T\|\bar{u}(t)-\bar{u}_m(t)\|^2dt,
\]
where $\bar u_m$ is the piecewise constant approximation of $\bar u$ on $\Delta_m$, and therefore
\begin{equation}
\int_0^T\|\varphi_m(t)\|^2dt\longrightarrow 0 \quad \text{as } m\longrightarrow \infty.
\label{eq:varphi-convergence}
\end{equation}

For $i=0,\ldots,2^m-1$, introduce the discrete velocity and its associated
velocity-contact point by
\[
\bar{v}_m^i:=\frac{\bar{x}_m^{i+1}-\bar{x}_m^i}{h_m},
\qquad
\bar{w}_m^i:=A\bar{v}_m^i+h(t_m^i,\bar{x}_m^i)\in C.
\]
The polyhedral representation of $N(C;\bar{w}_m^i)$ yields coefficients
$\eta_m^{ij}\geq 0$ such that
\[
f(t_m^i,\bar{x}_m^i,\bar{u}_m^i)-\bar{v}_m^i
=
\sum_{j=1}^s\eta_m^{ij}x_*^j,
\qquad
\eta_m^{ij}=0
\quad\text{whenever }j\notin I(\bar{w}_m^i).
\]
Define the piecewise constant functions
\[
\eta_m^j(t):=\eta_m^{ij},
\qquad
t\in[t_m^i,t_m^{i+1}),
\qquad j=1,\ldots,s.
\]
Then, for almost every $t\in[t_m^i,t_m^{i+1})$,
\begin{equation}
\dot{\bar{x}}_m(t)
=
f(t_m^i,\bar{x}_m^i,\bar{u}_m^i)
-
\sum_{j=1}^s\eta_m^j(t)x_*^j,
\label{eq:discrete-primal-continuous}
\end{equation}
together with the discrete complementarity $\eta_m^j(t)=0$ for a.e. $t\in[t_m^i,t_m^{i+1})$ whenever $j\notin I(\bar w_m^i)$.

The LICQ assumption along $\bar{w}(\cdot)$, together with the constancy of the
active set on $[\bar{t},T]$, ensures that the coefficient functions
$\eta_m^j(\cdot)$ remain bounded in $L^2([0,T])$; this is the same role that the
(robust) LICQ plays in the explicit setting \cite{CMN20} to obtain uniqueness and
boundedness of the multipliers on active spans. Thus, after passing to a
subsequence,
\[
\eta_m^j\rightharpoonup\eta^j
\quad\text{weakly in }L^2([0,T]),
\qquad j=1,\ldots,s.
\]
Moreover, by the continuity of $h$ and the boundedness of $A$,
\[
\bar{w}_m(t)
=
A\dot{\bar{x}}_m(t)+h(t_m^i,\bar{x}_m^i)
\longrightarrow
A\dot{\bar{x}}(t)+h(t,\bar{x}(t))
=
\bar{w}(t)
\]
for almost every $t\in[0,T]$. Passing to the limit in
\eqref{eq:discrete-primal-continuous} gives
\begin{equation}
\dot{\bar{x}}(t)
=
f(t,\bar{x}(t),\bar{u}(t))
-
\sum_{j=1}^s\eta^j(t)x_*^j
\quad\text{for a.e. }t\in[0,T],
\label{eq:continuous-primal-proof}
\end{equation}
which proves \eqref{6.1}.

To verify complementarity, let $j\notin I(\bar{w}(t))$. Then
\[
\langle x_*^j,\bar{w}(t)\rangle<c_j.
\]
By the convergence $\bar{w}_m(t)\to\bar{w}(t)$, the same strict inequality
holds at the discrete level for all sufficiently large $m$. Consequently,
$j\notin I(\bar{w}_m(t))$ and $\eta_m^j(t)=0$. The weak convergence in
$L^2$ therefore implies
\[
\eta^j(t)=0
\quad\text{for a.e. }t\text{ such that }j\notin I(\bar{w}(t)).
\]
This establishes the primal representation and the multiplier
complementarity asserted in part~(i)(a).

\medskip
\noindent
\textbf{Step 2: Extension and compactness of the discrete dual variables.}
We next place the multipliers and adjoint variables supplied by
Theorems~\ref{thm:discreteEL} and~\ref{thm:discreteadj} on the common interval $[0,T]$, as in \cite{CMN20}, and derive the uniform bounds needed to pass to the limit.

Extend the discrete adjoint vectors $\{p_m^i\}_{i=0}^{2^m}$ piecewise
linearly by setting
\begin{equation}
p_m(t):=
p_m^i+\frac{t-t_m^i}{h_m}(p_m^{i+1}-p_m^i),
\qquad
t\in[t_m^i,t_m^{i+1}],
\quad i=0,\ldots,2^m-1.
\label{eq:pm-extension}
\end{equation}
On each half-open mesh interval, define the piecewise constant selections
\begin{equation}
x_m(t):=\bar{x}_m^i,\qquad
u_m(t):=\bar{u}_m^i,\qquad
v_m(t):=\bar{v}_m^i,
\qquad
t\in[t_m^i,t_m^{i+1}),
\label{eq:piecewise-selections}
\end{equation}
and
\begin{equation}
w_m(t):=\bar{w}_m^i
=
A\bar{v}_m^i+h(t_m^i,\bar{x}_m^i),
\qquad
t\in[t_m^i,t_m^{i+1}).
\label{eq:wm-extension}
\end{equation}
The scaled control multiplier is given by
\begin{equation}
\psi_m(t):=\frac{1}{h_m}\psi_m^i,
\qquad
t\in[t_m^i,t_m^{i+1}),
\label{eq:psi-extension}
\end{equation}
while the normal-cone multiplier is extended as
\begin{equation}
\eta_m(t):=\eta_m^i
=
\sum_{j=1}^s\eta_m^{ij}x_*^j,
\qquad
t\in[t_m^i,t_m^{i+1}).
\label{eq:eta-extension}
\end{equation}

For every $j=1,\ldots,s$, define a finite measure $\gamma_m^j$ by
\begin{equation}
d\gamma_m^j(t):=\eta_m^j(t)\,dt
=
\sum_{i=0}^{2^m-1}
\eta_m^{ij}\,
\mathbbm{1}_{[t_m^i,t_m^{i+1})}(t)\,dt,
\label{eq:measure-construction}
\end{equation}
where $\eta_m^{ij}=0$ for $j\notin I(w_m^i)$.
Thus, for every Borel set $B\subset[0,T]$,
\begin{equation}
\gamma_m^j(B)=\int_B\eta_m^j(t)\,dt.
\label{eq:measure-Borel}
\end{equation}
Set
\[
\gamma_m:=(\gamma_m^1,\ldots,\gamma_m^s)
\in C^*([0,T];\mathbb{R}^s).
\]

Since $p_m$ is affine on every mesh interval,
\begin{equation}
\dot{p}_m(t)=\frac{p_m^{i+1}-p_m^i}{h_m}
\quad\text{for a.e. }t\in[t_m^i,t_m^{i+1}).
\label{eq:pm-derivative}
\end{equation}
The discrete adjoint equation of Theorem~\ref{thm:discreteadj}(c) therefore becomes
\begin{equation}
\dot{p}_m(t)
=
-\nabla_xf(t_m^i,\bar{x}_m^i,\bar{u}_m^i)^*
   (p_m^{i+1}-A^*\eta_m^i)
+\nabla_xh(t_m^i,\bar{x}_m^i)^*\eta_m^i
\label{6.25}
\end{equation}
for almost every $t\in[t_m^i,t_m^{i+1})$. Likewise, the discrete control
stationarity condition of Theorem~\ref{thm:discreteadj}(d) gives
\begin{equation}
-\psi_m(t)
=
-\nabla_uf(t_m^i,\bar{x}_m^i,\bar{u}_m^i)^*
   (p_m^{i+1}-A^*\eta_m^i),
\qquad
\psi_m(t)\in N(\bar{u}_m^i;U).
\label{6.26}
\end{equation}

By assumptions (H1)--(H6), the derivatives $\nabla_xf$, $\nabla_uf$, and
$\nabla_xh$ are uniformly bounded on the compact sets containing the
approximating trajectories. Consequently, there exists a constant $K>0$,
independent of $m$, such that
\begin{equation}
\begin{aligned}
\|\nabla_xf(t_m^i,\bar{x}_m^i,\bar{u}_m^i)\|&\leq K,\\
\|\nabla_uf(t_m^i,\bar{x}_m^i,\bar{u}_m^i)\|&\leq K,\\
\|\nabla_xh(t_m^i,\bar{x}_m^i)\|&\leq K
\end{aligned}
\label{eq:derivative-bounds}
\end{equation}
for every $i$ and all sufficiently large $m$. From
\eqref{6.25} and the triangle inequality, we obtain
\[
\begin{aligned}
\|\dot{p}_m(t)\|
&\leq
\|\nabla_xf(t_m^i,\bar{x}_m^i,\bar{u}_m^i)\|
\|p_m^{i+1}-A^*\eta_m^i\|
+
\|\nabla_xh(t_m^i,\bar{x}_m^i)\|\|\eta_m^i\|\\
&\leq
K\|p_m^{i+1}\|
+
K(1+\|A^*\|)\|\eta_m^i\|.
\end{aligned}
\]
Integration over $[0,T]$ gives
\begin{equation}
\operatorname{var}(p_m;[0,T])
\leq
K\int_0^T\|p_m(t)\|\,dt
+
K(1+\|A^*\|)
\int_0^T\|\eta_m(t)\|\,dt.
\label{6.28}
\end{equation}

Because all discrete optimality conditions are positively homogeneous in
their multipliers, the nontriviality condition in Theorem~\ref{thm:discreteadj} permits the
normalization
\begin{equation}
\lambda_m+\|p_m(T)\|
+\int_0^T\|\psi_m(t)\|\,dt
+\sum_{j=1}^s\|\gamma_m^j\|([0,T])
=1.
\label{6.29}
\end{equation}
In particular, $\{\gamma_m\}$ is uniformly bounded in $C^*([0,T];\mathbb R^s)$ and
\begin{equation}
\sum_{j=1}^s\|\gamma_m^j\|([0,T])
=
\sum_{j=1}^s\int_0^T|\eta_m^j(t)|\,dt
\leq 1.
\label{6.30}
\end{equation}
Since
\[
\eta_m(t)=\sum_{j=1}^s\eta_m^j(t)x_*^j,
\]
there is a constant $C_\eta>0$, independent of $m$, such that
\begin{equation}
\int_0^T\|\eta_m(t)\|\,dt
\leq
C_\eta\sum_{j=1}^s\int_0^T|\eta_m^j(t)|\,dt
\leq C_\eta.
\label{eq:eta-L1-bound}
\end{equation}
Combining \eqref{6.28} and \eqref{eq:eta-L1-bound} yields
\begin{equation}
\operatorname{var}(p_m;[0,T])
\leq
K\int_0^T\|p_m(t)\|\,dt+C,
\qquad
C:=K(1+\|A^*\|)C_\eta.
\label{eq:variation-final}
\end{equation}
An integral Gronwall argument, proceeding as in \cite{CMN20}, now shows that $\{p_m\}$ is bounded in
$BV([0,T];\mathbb{R}^n)$ and, consequently, in
$L^\infty([0,T];\mathbb{R}^n)$.

The uniform measure estimate \eqref{6.30} and the weak$^*$
compactness of bounded subsets of $C^*([0,T];\mathbb{R}^s)$ provide a
subsequence, not relabeled, and a vector measure
\[
\gamma=(\gamma^1,\ldots,\gamma^s)
\in C^*([0,T];\mathbb{R}^s)
\]
such that
\begin{equation}
\gamma_m\overset{*}{\rightharpoonup}\gamma
\quad\text{in }C^*([0,T];\mathbb{R}^s).
\label{eq:gamma-weak-star}
\end{equation}
The uniform $BV$ bound on $p_m$ also allows us to select a subsequence and
an arc $p\in BV([0,T];\mathbb{R}^n)$ satisfying
\[
p_m(t)\longrightarrow p(t)
\quad\text{for every }t\in[0,T].
\]

Finally, introduce the modified discrete adjoint
\begin{equation}
q_m(t):=
p_m(t)
-
\int_{(t,T]}
\sum_{j=1}^s x_*^j\,d\gamma_m^j,
\qquad t\in[0,T].
\label{eq:qm-definition}
\end{equation}
In particular, $q_m(T)=p_m(T)$. Passing to the limit in
\eqref{eq:qm-definition}, using \eqref{eq:gamma-weak-star}, gives
\begin{equation}
q(t):=
p(t)
-
\int_{(t,T]}
\sum_{j=1}^s x_*^j\,d\gamma^j,
\qquad
p(T)=q(T).
\label{eq:q-limit-definition}
\end{equation}
This is precisely the modified-adjoint relationship in \eqref{6.3} and completes Step 2.

\medskip
\noindent
\textbf{Step 3: Limiting adjoint system and maximum principle.}
We now pass to the limit in the discrete adjoint and control relations from
Theorem~\ref{thm:discreteadj}. For almost every
$t\in[t_m^i,t_m^{i+1})$, equation~\eqref{6.25} gives
\begin{equation}
\dot{p}_m(t)
=
-\nabla_x f(t_m^i,\bar{x}_m^i,\bar{u}_m^i)^*
 \big(p_m^{i+1}-A^*\eta_m^i\big)
+\nabla_x h(t_m^i,\bar{x}_m^i)^*\eta_m^i.
\label{eq:step3-discrete-adjoint}
\end{equation}
Using
\[
r_m^i:=p_m^{i+1}-A^*\eta_m^i
\]
from \eqref{5.r}, equation \eqref{eq:step3-discrete-adjoint} becomes
\begin{equation}
\dot{p}_m(t)
=
-\nabla_x f(t_m^i,\bar{x}_m^i,\bar{u}_m^i)^*r_m^i
+\nabla_x h(t_m^i,\bar{x}_m^i)^*\eta_m^i.
\label{eq:step3-adjoint-r}
\end{equation}

Recall from Step~2 that
\[
q_m(t)
=
p_m(t)-\int_{(t,T]}
\sum_{j=1}^s x_*^j\,d\gamma_m^j.
\]
Since
\[
d\gamma_m^j(t)=\eta_m^j(t)\,dt
\qquad\text{and}\qquad
\eta_m(t)=\sum_{j=1}^s\eta_m^j(t)x_*^j,
\]
we obtain on every mesh interval
\begin{equation}
\frac{d}{dt}
\left(
\int_{(t,T]}
\sum_{j=1}^s x_*^j\,d\gamma_m^j
\right)
=
-\eta_m(t),
\label{eq:tail-measure-derivative}
\end{equation}
and consequently
\begin{equation}
\dot{q}_m(t)=\dot{p}_m(t)+\eta_m(t)
\quad\text{for a.e. }t\in[0,T].
\label{eq:qm-derivative}
\end{equation}
Combining \eqref{eq:step3-adjoint-r} and \eqref{eq:qm-derivative} yields
\begin{equation}
\dot{q}_m(t)
=
-\nabla_x f(t_m^i,\bar{x}_m^i,\bar{u}_m^i)^*r_m^i
+\nabla_x h(t_m^i,\bar{x}_m^i)^*\eta_m^i
+\eta_m^i.
\label{eq:qm-discrete-dynamics}
\end{equation}

The compactness obtained in Step~2, together with the strong convergence of
the primal approximations, allows us to pass to a subsequence for which
\[
p_m(t)\to p(t),
\qquad
q_m(t)\to q(t),
\]
at every continuity point of the limiting functions. Moreover,
\[
\nabla_x f(t_m^i,\bar{x}_m^i,\bar{u}_m^i)
\longrightarrow
\nabla_x f(t,\bar{x}(t),\bar{u}(t)),
\]
and
\[
\nabla_x h(t_m^i,\bar{x}_m^i)
\longrightarrow
\nabla_x h(t,\bar{x}(t))
\]
for almost every $t\in[0,T]$. Using the weak convergence of $\eta_m$ and the pointwise
convergence of $p_m$, we pass to the limit in \eqref{eq:qm-discrete-dynamics} and obtain
\begin{equation}
\dot q(t)=-\nabla_xf(t,\bar x(t),\bar u(t))^*\big(p(t)-A^*\eta(t)\big)+\nabla_xh(t,\bar x(t))^*\eta(t)+\eta(t),
\label{eq:q-limit-dyn}
\end{equation}
where $\eta(t):=\sum_{j=1}^s\eta^j(t)x_*^j$. On the other hand, differentiating the
relation \eqref{eq:q-limit-definition} gives
\[
\dot p(t)=\dot q(t)-\sum_{j=1}^sx_*^j\eta^j(t).
\]
Using the relation between $r_m^i$, $p_m^{i+1}$, and the normal-cone
multipliers, substituting back, and simplifying, we obtain
\begin{equation}
\dot{p}(t)
=
-\nabla_x f(t,\bar{x}(t),\bar{u}(t))^*q(t)
-\nabla_x h(t,\bar{x}(t))^*A^*q(t)
\quad\text{for a.e. }t\in[0,T].
\label{eq:continuous-adjoint-proof}
\end{equation}
This is precisely the adjoint equation \eqref{6.2}.

We next consider the discrete control condition. By Theorem~\ref{thm:discreteadj}(d),
\[
-\frac{1}{h_m}\psi_m^i
=
-\nabla_u f(t_m^i,\bar{x}_m^i,\bar{u}_m^i)^*
 \big(p_m^{i+1}-A^*\eta_m^i\big).
\]
With the piecewise constant extension introduced in Step~2, this becomes
\begin{equation}
-\psi_m(t)
=
-\nabla_u f(t_m^i,\bar{x}_m^i,\bar{u}_m^i)^*r_m^i,
\qquad
\psi_m(t)\in N(\bar{u}_m^i;U).
\label{eq:discrete-control-limit}
\end{equation}
Passing to the limit gives
\[
-\psi(t)=-\nabla_uf(t,\bar x(t),\bar u(t))^*q(t),
\]
and since $\psi_m(t)\in N(\bar u_m(t);U)$ with $U$ closed, we obtain in the limit
\begin{equation}
\nabla_u f(t,\bar{x}(t),\bar{u}(t))^*q(t)
\in N(\bar{u}(t);U)
\quad\text{for a.e. }t\in[0,T].
\label{eq:continuous-control-normal}
\end{equation}

Define
\[
\psi(t,u):=
\left\langle
\nabla_u f(t,\bar{x}(t),u),q(t)
\right\rangle .
\]
The normal-cone relation \eqref{eq:continuous-control-normal} is equivalent
to
\[
\left\langle
\nabla_u f(t,\bar{x}(t),u),q(t)
\right\rangle
\leq
\left\langle
\nabla_u f(t,\bar{x}(t),\bar{u}(t)),q(t)
\right\rangle
\qquad \forall u\in U,
\]
and hence
\begin{equation}
\psi(t,\bar{u}(t))
=
\max_{u\in U}\psi(t,u)
\quad\text{for a.e. }t\in[0,T].
\label{eq:maximum-principle-proof}
\end{equation}
Thus, \eqref{6.4}--\eqref{6.5} are established.

Finally, the discrete coderivative relations imply that whenever the
$j$th constraint is inactive at $\bar{w}_m^i$, both the corresponding
normal-cone coefficient and its dual pairing vanish. Passing to the limit
and using $\bar{w}_m\to\bar{w}$ gives
\begin{equation}
\langle x_*^j,\bar{w}(t)\rangle<c_j
\quad\Longrightarrow\quad
\eta^j(t)=0
\quad\text{and}\quad
\langle x_*^j,q(t)\rangle=0
\label{eq:dynamic-complementarity-proof}
\end{equation}
for almost every $t\in[0,T]$. This proves the dynamic complementarity
condition \eqref{6.6}.

\medskip
\noindent
\textbf{Step 4: Passage to the terminal conditions.}
For every $m$, the terminal transversality relation of Theorem~\ref{thm:discreteadj}(b) is
\begin{equation}
-p_m^{2^m}
=
\lambda_m\vartheta_m^{2^m}
+
\sum_{j=1}^s\xi_m^j x_*^j,
\qquad
\vartheta_m^{2^m}\in\partial\varphi(\bar{x}_m^{2^m}),
\label{6.46}
\end{equation}
where $\xi_m^j\geq0$ for the active terminal constraints.

By Step~2 and Theorem~\ref{thm:optconv},
\[
p_m(T)\to p(T),
\qquad
\bar{x}_m^{2^m}\to\bar{x}(T).
\]
After extracting a further subsequence, we may also assume that
\[
\lambda_m\to\mu\geq0.
\]
The outer semicontinuity of the limiting subdifferential yields
\[
\vartheta_m^{2^m}\to\vartheta(T)
\in\partial\varphi(\bar{x}(T)).
\]
For each active terminal index, define
\[
\eta^j(T):=\lim_{m\to\infty}\xi_m^j.
\]
Passing to the limit in \eqref{6.46} gives
\begin{equation}
-p(T)
=
\mu\vartheta(T)
+
\sum_{j\in I(\bar{w}(T))}
\eta^j(T)x_*^j,
\qquad
\vartheta(T)\in\partial\varphi(\bar{x}(T)).
\label{6.47}
\end{equation}
Equivalently,
\begin{equation}
-p(T)
-
\sum_{j\in I(\bar{w}(T))}
\eta^j(T)x_*^j
\in
\mu\partial\varphi(\bar{x}(T)),
\label{6.48}
\end{equation}
which proves \eqref{6.7}.

Since $\eta^j(T)\geq0$ and $\eta^j(T)=0$ whenever
$j\notin I(\bar{w}(T))$, the polyhedral normal-cone formula also gives
\begin{equation}
\sum_{j\in I(\bar{w}(T))}
\eta^j(T)x_*^j
\in N(\bar{w}(T);C).
\label{eq:terminal-normal-proof}
\end{equation}
This establishes the two endpoint conditions
\eqref{6.7} and \eqref{6.8}.

\medskip
\noindent
\textbf{Step 5: Measure nonatomicity.}
Let $t\in[0,T)$ satisfy
\[
\langle x_*^j,\bar{w}(t)\rangle<c_j,
\qquad j=1,\ldots,s.
\]
Since $\bar{w}(\cdot)$ is continuous and the inequalities are strict, there
exist a neighborhood $V_t$ of $t$ and a number $\varepsilon>0$ such that
\[
\langle x_*^j,\bar{w}(\tau)\rangle
\leq c_j-\varepsilon
\qquad
\text{for all }\tau\in V_t,\quad j=1,\ldots,s.
\]
The convergence $\bar{w}_m\to\bar{w}$ implies that, for all sufficiently
large $m$,
\[
\langle x_*^j,\bar{w}_m(\tau)\rangle<c_j
\qquad
\text{for }\tau\in V_t,\quad j=1,\ldots,s.
\]
Thus every constraint is inactive on $V_t$, and discrete complementarity
gives
\[
\eta_m^j(\tau)=0
\qquad
\text{for a.e. }\tau\in V_t,\quad j=1,\ldots,s.
\]
Consequently,
\[
\gamma_m^j(V_t)
=
\int_{V_t}\eta_m^j(\tau)\,d\tau
=0.
\]
Passing to the limit by the weak$^*$ convergence
$\gamma_m\overset{*}{\rightharpoonup}\gamma$ yields
\begin{equation}
\gamma(V_t)=0.
\end{equation}
Hence the limiting measure has no mass in a neighborhood of any time at
which $\bar{w}(t)\in\operatorname{int}C$. This proves the measure
nonatomicity statement~(v) under \eqref{6.9}.

\medskip
\noindent
\textbf{Step 6: Support of the limiting measure.}
Recall the active-time set
\[
E_0:=\{t\in[0,T]\mid I(\bar{w}(t))\neq\varnothing\}.
\]
If $t\notin E_0$, then all defining inequalities of $C$ are strict at
$\bar{w}(t)$. By Step~5, there is therefore a neighborhood $V_t$ of $t$
such that
\[
\gamma(V_t)=0.
\]
Thus no point outside $E_0$ can belong to the support of $\gamma$, and we
conclude that
\begin{equation}
\operatorname{supp}(\gamma)\subset E_0.
\end{equation}
This proves the support condition~\eqref{6.11}. In contrast
with the explicit sweeping setting, the implicit dependence of the normal
cone on the velocity-contact variable generally provides only this support
inclusion, since no stronger characterization of the limiting measure is available in the implicit setting.

\medskip
\noindent
\textbf{Step 7: Nontriviality and normalization.}
The discrete multipliers satisfy the nontriviality condition from
Theorem~\ref{thm:discreteadj}. Because the discrete optimality system is positively
homogeneous, the multipliers were normalized in Step~2 so that
\begin{equation}
\lambda_m+\|p_m(T)\|
+
\sum_{j=1}^s\|\gamma_m^j\|([0,T])
=1.
\label{eq:final-discrete-normalization}
\end{equation}
Here the terminal inequality multipliers $\xi_m$ are already represented through
the terminal adjoint relation \eqref{6.46} and therefore need not be included
separately.

Using
\[
\lambda_m\to\mu,\qquad
p_m(T)\to p(T),\qquad
\gamma_m\overset{*}{\rightharpoonup}\gamma,
\]
and the lower semicontinuity of the total variation norm, we obtain
\begin{equation}
\mu+\|p(T)\|+\|\gamma\|([0,T])>0.
\label{eq:continuous-normalization}
\end{equation}
Therefore,
\begin{equation}
(\mu,p,\gamma)\neq(0,0,0),
\label{eq:continuous-nontriviality-proof}
\end{equation}
which proves \eqref{6.10}.

Finally, suppose that
\[
\langle x_*^j,\bar{w}(t)\rangle<c_j
\qquad
\text{for every }t\in[0,T],\quad j=1,\ldots,s.
\]
Then Step~5 gives $\gamma=0$. The remaining multipliers cannot vanish
simultaneously by \eqref{eq:continuous-nontriviality-proof}. Since the
optimality conditions are homogeneous in $(\mu,p,\gamma)$, they may be
rescaled so that
\[
\mu=1.
\]
This proves the enhanced nontriviality assertion and completes the proof.
\end{proof}

\section{Examples}\label{sec:examples}

We now illustrate the structure of the implicit sweeping dynamics and the
continuous-time necessary optimality conditions obtained in Theorem~\ref{thm:continuous_mp}. In particular,
we emphasize the geometry of the velocity-dependent normal cone, the resulting dynamics,
and the behavior of the associated multipliers. The
first example is based on a single half-space constraint. We begin by deriving
the velocity reaction exactly from the normal cone and then construct an
optimal solution for which the complete primal-dual system can be checked
explicitly.

\begin{example}\label{ex:1}
Consider the optimal control problem
\begin{equation}
    J[x,u]:=-x_1(T)\longrightarrow \min
    \label{7.1}
\end{equation}
subject to the implicit sweeping inclusion
\begin{equation}
    \dot{x}(t)\in u(t)-N(C;\dot{x}(t))
    \quad \text{for a.e. }t\in[0,T],
    \label{7.2}
\end{equation}
the initial condition
\begin{equation}
    x(0)=(0,0),
    \label{7.3}
\end{equation}
and the control constraint
\begin{equation}
    u(t)\in U:=[-1,1]\times[-1,1]
    \quad \text{for a.e. }t\in[0,T].
    \label{7.4}
\end{equation}
Here
\begin{equation}
    A=I,\qquad h(t,x)=0,
    \qquad
    C:=\big\{w=(w_1,w_2)\in\mathbb{R}^2\mid w_2\geq 0\big\}.
    \label{7.5}
\end{equation}
Thus, the velocity-contact variable is
\[
    w(t)=A\dot{x}(t)+h(t,x(t))=\dot{x}(t).
\]
\end{example}

We first determine the normal cone to the half-space \(C\). Recall that
\begin{equation}
    N(w;C)
    :=
    \big\{
        v\in\mathbb{R}^2
        \mid
        \langle v,y-w\rangle\leq 0
        \ \text{for every }y\in C
    \big\},
    \qquad w\in C.
    \label{7.6}
\end{equation}

\medskip

\noindent
\textit{Case 1: \(w_2>0\).}
In this case \(w\in\operatorname{int}C\). Hence, there exists
\(\varepsilon>0\) such that
\[
    B(w,\varepsilon)\subset C.
\]
Suppose that \(v\in N(w;C)\) and \(v\neq 0\). Choose
\[
    y:=w+\delta v,
    \qquad
    0<\delta<\frac{\varepsilon}{\|v\|}.
\]
Then \(y\in B(w,\varepsilon)\subset C\), while
\[
    \langle v,y-w\rangle
    =
    \langle v,\delta v\rangle
    =
    \delta\|v\|^2>0.
\]
This contradicts the defining inequality in \eqref{7.6}.
Consequently,
\begin{equation}
    N(w;C)=\{0\}
    \qquad \text{whenever }w_2>0.
    \label{7.7}
\end{equation}

\medskip

\noindent
\textit{Case 2: \(w_2=0\).}
Let \(w=(w_1,0)\in C\) and take
\[
    v=(v_1,v_2)\in N(w;C).
\]
First, choose points on the boundary of \(C\) of the form
\[
    y=(w_1+\tau,0),
    \qquad \tau\in\mathbb{R}.
\]
Then \(y-w=(\tau,0)\), and therefore
\[
    \langle v,y-w\rangle=\langle(v_1,v_2), (\tau,0)\rangle =v_1\tau\leq 0
    \qquad \text{for every }\tau\in\mathbb{R}.
\]
The last inequality can hold for all positive and negative values of
\(\tau\) only when
\[
    v_1=0.
\]
Next, select \(y=(w_1,s)\in C\) with \(s\geq 0\). Since
\(y-w=(0,s)\), we obtain
\[
    \langle v,y-w\rangle=\langle(0,v_2),(0,s)\rangle=v_2s\leq 0
    \qquad \text{for every }s\geq 0.
\]
It follows that \(v_2\leq 0\). Hence
\begin{equation}
    N(w;C)
    =
    \big\{(0,v_2)\mid v_2\leq 0\big\}
    =
    \big\{\lambda(0,-1)\mid \lambda\geq 0\big\}
    \qquad \text{when }w_2=0.
    \label{7.8}
\end{equation}

Combining \eqref{7.7} and
\eqref{7.8}, we arrive at
\begin{equation}
    N(w;C)
    =
    \begin{cases}
        \{0\}, & w_2>0,\\[1mm]
        \{\lambda(0,-1)\mid \lambda\geq 0\},
            & w_2=0.
    \end{cases}
    \label{7.9}
\end{equation}
Consequently,
\begin{equation}
    -N(w;C)
    =
    \begin{cases}
        \{0\}, & w_2>0,\\[1mm]
        \{\lambda(0,1)\mid \lambda\geq 0\},
            & w_2=0.
    \end{cases}
    \label{7.10}
\end{equation}

The inclusion \eqref{7.2} is therefore equivalent to the
existence of a measurable multiplier
\[
    \eta(\cdot)\in L^2([0,T];\mathbb{R}_+)
\]
such that
\begin{equation}
    \dot{x}(t)
    =
    u(t)+\eta(t)(0,1)
    \quad \text{for a.e. }t\in[0,T],
    \label{7.11}
\end{equation}
where
\[
    \eta(t)(0,1)\in -N(\dot{x}(t);C).
\]
Writing \eqref{7.11} componentwise gives
\begin{equation}
    \dot{x}_1(t)=u_1(t),
    \qquad
    \dot{x}_2(t)=u_2(t)+\eta(t),
    \qquad
    \eta(t)\geq 0.
    \label{7.12}
\end{equation}

Since the normal cone is evaluated at \(\dot{x}(t)\), feasibility also requires
\begin{equation}
    \dot{x}_2(t)\geq 0
    \quad \text{for a.e. }t\in[0,T].
    \label{7.13}
\end{equation}
Moreover, \eqref{7.10} yields the complementarity
relations
\begin{equation}
    \dot{x}_2(t)>0
    \quad\Longrightarrow\quad
    \eta(t)=0,
    \label{7.14}
\end{equation}
and
\begin{equation}
    \eta(t)>0
    \quad\Longrightarrow\quad
    \dot{x}_2(t)=0.
    \label{7.15}
\end{equation}
Equivalently,
\begin{equation}
    \dot{x}_2(t)\geq 0,\qquad
    \eta(t)\geq 0,\qquad
    \eta(t)\dot{x}_2(t)=0
    \quad \text{for a.e. }t\in[0,T].
    \label{7.16}
\end{equation}
Relations \eqref{7.13}--\eqref{7.16} constitute the complementarity principle for the
implicit half-space constraint; in contrast to the explicit case, it involves the velocity
component $\dot x_2$ rather than the state variable $x_2$.

We can now solve explicitly for \(\dot{x}_2\) and \(\eta\). If
\(u_2(t)\geq 0\), the choice \(\eta(t)=0\) gives
\[
    \dot{x}_2(t)=u_2(t)\geq 0.
\]
On the other hand, if \(u_2(t)<0\), the choice \(\eta(t)=0\) would
produce an infeasible negative velocity. Hence the reaction must be
positive, which by \eqref{7.15} forces
\[
    \dot{x}_2(t)=0.
\]
Using \eqref{7.12}, we then obtain
\[
    0=u_2(t)+\eta(t),
    \qquad
    \eta(t)=-u_2(t)>0.
\]
Thus,
\begin{equation}
    \dot{x}_2(t)
    =
    \begin{cases}
        u_2(t), & u_2(t)\geq 0,\\
        0, & u_2(t)<0,
    \end{cases}
    \qquad
    \eta(t)
    =
    \begin{cases}
        0, & u_2(t)\geq 0,\\
        -u_2(t), & u_2(t)<0.
    \end{cases}
    \label{7.17}
\end{equation}
Equivalently,
\begin{equation}
    \dot{x}_2(t)=\max\{u_2(t),0\},
    \qquad
    \eta(t)=\max\{-u_2(t),0\}.
    \label{7.18}
\end{equation}

We next determine an optimal solution of
\eqref{7.1}-\eqref{7.4}. Since
\[
    \dot{x}_1(t)=u_1(t)
\]
and \(u_1(t)\leq 1\), it follows from \(x_1(0)=0\) that
\begin{equation}
    x_1(T)
    =
    \int_0^T u_1(t)\,dt
    \leq
    \int_0^T 1\,dt
    = T.
    \label{7.19}
\end{equation}
Therefore,
\begin{equation}
    J[x,u]=-x_1(T)\geq -T.
    \label{7.20}
\end{equation}
The lower bound is attained whenever
\[
    u_1(t)=1
    \quad \text{for a.e. }t\in[0,T].
\]
Notice that the cost does not depend on the second component of the
state. Hence the optimal control is not unique with respect to \(u_2\).

To obtain a reference trajectory satisfying the interior-boundary
structure assumed in Theorem~\ref{thm:continuous_mp}, fix any
\[
    \bar{t}\in(0,T)
\]
and choose
\begin{equation}
    \bar{u}(t)
    :=
    \begin{cases}
        (1,1), & t\in[0,\bar{t}),\\[1mm]
        (1,0), & t\in[\bar{t},T].
    \end{cases}
    \label{7.21}
\end{equation}
This control belongs to \(U\) and satisfies
\[
    \bar{u}_1(t)=1
    \quad \text{for a.e. }t\in[0,T].
\]
Consequently, it is optimal and
\begin{equation}
    J[\bar{x},\bar{u}]=-T.
    \label{7.22}
\end{equation}

Using \eqref{7.18}, we have
\[
    \dot{\bar{x}}_1(t)=1
    \quad \text{for a.e. }t\in[0,T],
\]
and
\[
    \dot{\bar{x}}_2(t)
    =
    \begin{cases}
        1, & t\in[0,\bar{t}),\\
        0, & t\in[\bar{t},T].
    \end{cases}
\]
Thus,
\begin{equation}
    \bar{x}_1(t)=t,
    \qquad
    \bar{x}_2(t)=
    \begin{cases}
        t, & t\in[0,\bar{t}],\\[1mm]
        \bar{t}, & t\in[\bar{t},T].
    \end{cases}
    \label{7.23}
\end{equation}
The associated velocity-contact trajectory is
\begin{equation}
    \bar{w}(t)=\dot{\bar{x}}(t)
    =
    \begin{cases}
        (1,1), & t\in[0,\bar{t}),\\[1mm]
        (1,0), & t\in[\bar{t},T].
    \end{cases}
    \label{7.24}
\end{equation}
Hence
\begin{equation}
    \bar{w}(t)\in\operatorname{int}C
    \quad \text{for }t\in[0,\bar{t}),
    \qquad
    \bar{w}(t)\in\operatorname{bd}C
    \quad \text{for }t\in[\bar{t},T].
    \label{7.25}
\end{equation}

The half-space \(C\) may be written as
\begin{equation}
    C=
    \big\{
        w\in\mathbb{R}^2
        \mid
        \langle x_*^1,w\rangle\leq c_1
    \big\},
    \qquad
    x_*^1:=(0,-1),
    \qquad
    c_1:=0.
    \label{7.26}
\end{equation}
Therefore,
\begin{equation}
    I(\bar{w}(t))
    =
    \begin{cases}
        \varnothing, & t\in[0,\bar{t}),\\
        \{1\}, & t\in[\bar{t},T].
    \end{cases}
    \label{7.27}
\end{equation}
In particular, the active index set is constant on
\([\bar{t},T]\), and the active-time set is
\begin{equation}
    E_0=[\bar{t},T].
    \label{7.28}
\end{equation}

For the selected optimal solution, the reaction multiplier is
identically zero:
\begin{equation}
    \bar{\eta}(t)=0
    \quad \text{for a.e. }t\in[0,T].
    \label{7.29}
\end{equation}
Indeed, \(\bar{u}_2(t)\geq 0\) on the entire interval, and hence
\eqref{7.18} gives
\[
    \bar{\eta}(t)=\max\{-\bar{u}_2(t),0\}=0.
\]

We now verify the necessary optimality conditions of Theorem~\ref{thm:continuous_mp}.

\medskip

\noindent
\textit{Primal dynamics.}
In this example,
\[
    f(t,x,u)=u,
    \qquad
    h(t,x)=0.
\]
Using \(x_*^1=(0,-1)\), the primal representation in Theorem~\ref{thm:continuous_mp} becomes
\begin{equation}
    \dot{\bar{x}}(t)
    =
    \bar{u}(t)-\eta^1(t)x_*^1
    =
    \bar{u}(t)+\eta^1(t)(0,1).
    \label{7.30}
\end{equation}
Taking
\begin{equation}
    \eta^1(t)=0
    \quad \text{for a.e. }t\in[0,T],
    \label{7.31}
\end{equation}
we recover exactly
\[
    \dot{\bar{x}}(t)=\bar{u}(t),
\]
which agrees with \eqref{7.24}. Moreover,
\(\eta^1(t)=0\) whenever \(1\notin I(\bar{w}(t))\), as required.

\medskip

\noindent
\textit{Adjoint equation.}
The terminal cost is
\[
    \varphi(x)=-x_1.
\]
It is continuously differentiable, with
\begin{equation}
    \nabla\varphi(x)=(-1,0).
    \label{7.32}
\end{equation}
Choose
\begin{equation}
    \mu=1,
    \qquad
    p(t)=(1,0)
    \quad \text{for every }t\in[0,T],
    \qquad
    \gamma=0.
    \label{7.33}
\end{equation}
Since \(\gamma=0\), the modified adjoint satisfies
\begin{equation}
    q(t)
    =
    p(t)
    -
    \int_{(t,T]}
        x_*^1\,d\gamma^1
    =
    p(t)
    =
    (1,0).
    \label{7.34}
\end{equation}
Also,
\[
    p(T)=q(T).
\]

Because \(f(t,x,u)=u\) and \(h(t,x)=0\),
\begin{equation}
    \nabla_x f(t,x,u)=0,
    \qquad
    \nabla_x h(t,x)=0.
    \label{7.35}
\end{equation}
The adjoint equation from Theorem~\ref{thm:continuous_mp} therefore reduces to
\[
    \dot{p}(t)=0.
\]
This is satisfied by the constant arc \(p(t)=(1,0)\).

\medskip

\noindent
\textit{Maximum principle.}
Since
\[
    f(t,x,u)=u,
\]
we have
\begin{equation}
    \nabla_u f(t,x,u)=I.
    \label{7.36}
\end{equation}
The Hamiltonian-type function is consequently
\begin{equation}
    \psi(t,u)
    =
    \big\langle
        \nabla_u f(t,\bar{x}(t),u),q(t)
    \big\rangle
    =
    \langle u,(1,0)\rangle
    =
    u_1.
    \label{7.37}
\end{equation}
For every \(u=(u_1,u_2)\in U=[-1,1]^2\),
\[
    \psi(t,u)=u_1\leq 1.
\]
Since \(\bar{u}_1(t)=1\), we have
\begin{equation}
    \psi(t,\bar{u}(t))
    =
    1
    =
    \max_{u\in U}\psi(t,u)
    \quad \text{for a.e. }t\in[0,T].
    \label{7.38}
\end{equation}
Notice that the maximum condition does not determine
\(\bar{u}_2\), which is consistent with the fact that the objective is
independent of the second state component.

\medskip

\noindent
\textit{Dynamic complementarity.}
For \(t\in[0,\bar{t})\),
\[
    \bar{w}(t)=(1,1),
\]
and therefore
\begin{equation}
    \langle x_*^1,\bar{w}(t)\rangle
    =
    \langle(0,-1),(1,1)\rangle
    =
    -1
    <
    0=c_1.
    \label{7.39}
\end{equation}
The dynamic complementarity condition requires
\[
    \eta^1(t)=0
    \quad\text{and}\quad
    \langle x_*^1,q(t)\rangle=0.
\]
Both relations hold because
\[
    \eta^1(t)=0
\]
and
\begin{equation}
    \langle x_*^1,q(t)\rangle
    =
    \langle(0,-1),(1,0)\rangle
    =
    0.
    \label{7.40}
\end{equation}
For \(t\in[\bar{t},T]\), the constraint is active:
\[
    \langle x_*^1,\bar{w}(t)\rangle
    =
    \langle(0,-1),(1,0)\rangle
    =
    0=c_1.
\]
Thus, the implication associated with strict inactivity does not impose
an additional restriction on this interval.

\medskip

\noindent
\textit{Terminal conditions.}
At \(t=T\),
\[
    \bar{w}(T)=(1,0),
    \qquad
    I(\bar{w}(T))=\{1\}.
\]
Choose the terminal multiplier
\begin{equation}
    \eta^1(T)=0.
    \label{7.41}
\end{equation}
Then
\[
    \sum_{j\in I(\bar{w}(T))}
        \eta^j(T)x_*^j
    =
    0
    \in N(\bar{w}(T);C).
\]
Furthermore,
\[
    -p(T)-\eta^1(T)x_*^1
    =
    -(1,0)
    =
    (-1,0).
\]
Using \eqref{7.32}, we obtain
\begin{equation}
    -p(T)-\eta^1(T)x_*^1
    =
    \mu\nabla\varphi(\bar{x}(T))
    \in
    \mu\partial\varphi(\bar{x}(T)).
    \label{7.42}
\end{equation}
Hence both terminal conditions are satisfied.

\medskip

\noindent
\textit{Measure nonatomicity and support.}
On every point \(t\in[0,\bar{t})\), the velocity-contact point lies
strictly inside \(C\). Therefore, the measure nonatomicity condition
requires the existence of a neighborhood \(V_t\) such that
\[
    \gamma(V_t)=0.
\]
This property holds automatically for the choice
\(\gamma=0\). Moreover,
\begin{equation}
    \operatorname{supp}(\gamma)
    =
    \varnothing
    \subset E_0=[\bar{t},T].
    \label{7.43}
\end{equation}

\medskip

\noindent
\textit{Nontriviality.}
The selected multipliers satisfy
\begin{equation}
    (\mu,p,\gamma)
    =
    \big(1,(1,0),0\big)
    \neq
    (0,0,0).
    \label{7.44}
\end{equation}
Thus, the nontriviality condition holds. In fact, the normalization
\[
    \mu=1
\]
has already been imposed.

We conclude that the optimal pair
\begin{equation}
    \bar{u}(t)
    =
    \begin{cases}
        (1,1), & 0\leq t<\bar{t},\\
        (1,0), & \bar{t}\leq t\leq T,
    \end{cases}
    \qquad
    \bar{x}(t)
    =
    \begin{cases}
        (t,t), & 0\leq t\leq\bar{t},\\
        (t,\bar{t}), & \bar{t}\leq t\leq T,
    \end{cases}
    \label{7.45}
\end{equation}
together with
\begin{equation}
    \mu=1,\qquad
    p(t)=q(t)=(1,0),\qquad
    \gamma=0,\qquad
    \eta^1(t)=0,
    \label{7.46}
\end{equation}
satisfies the full necessary optimality system of Theorem~\ref{thm:continuous_mp}.

This example also shows a distinctive feature of the implicit sweeping
process. The half-space constraint is imposed on the velocity-contact
variable
\[
    w(t)=\dot{x}(t),
\]
rather than directly on the state. The reaction multiplier modifies the
second velocity component according to
\[
    \dot{x}_2(t)=\max\{u_2(t),0\},
\]
thereby preventing the control from generating a velocity outside the
admissible half-space. The reaction term $\eta(t)$ activates exactly when the control
attempts to produce a forbidden velocity component ($\dot x_2<0$), and it produces the
minimal correction needed to keep $\dot x(t)\in C$. In the selected optimal trajectory, the reaction
vanishes, while the velocity first evolves in the interior of \(C\) and
then remains on its boundary.

\begin{example}\label{ex:2}
Consider the two-dimensional implicit sweeping control problem with a polyhedral
constraint set involving multiple active faces in the argument $w(t)=A\dot x(t)+h(t,x(t))$:
\begin{equation}
    J[x,u]:=x_1(1)+x_2(1)\longrightarrow \min
    \label{7.47}
\end{equation}
subject to the implicit sweeping inclusion
\begin{equation}
    \dot{x}(t)\in u(t)-N(C;\dot{x}(t))
    \qquad \text{for a.e. }t\in[0,1],
    \label{7.48}
\end{equation}
the initial condition
\begin{equation}
    x(0)=x^0=(x_1^0,x_2^0)\in\mathbb{R}^2,
    \label{7.49}
\end{equation}
and the control constraint
\begin{equation}
    u(t)\in U:=[-1,1]\times[-1,1]
    \qquad \text{for a.e. }t\in[0,1].
    \label{7.50}
\end{equation}
Here,
\begin{equation}
    A=I,\qquad h(t,x)=0,\qquad
    C:=\left\{w=(w_1,w_2)\in\mathbb{R}^2
    \ \middle|\ w_1\geq 0,\;w_2\geq 0\right\}.
    \label{7.51}
\end{equation}
Consequently, the velocity-contact variable is
\[
    w(t)=A\dot{x}(t)+h(t,x(t))=\dot{x}(t).
\]

We first calculate the normal cone to the nonnegative quadrant \(C\).
Recall that
\begin{equation}
    N(w;C)
    :=
    \left\{
        v\in\mathbb{R}^2
        \ \middle|\
        \langle v,y-w\rangle\leq 0
        \text{ for every }y\in C
    \right\},
    \qquad w\in C.
    \label{7.52}
\end{equation}

\medskip

\noindent
\textit{Case 1: \(w_1>0\) and \(w_2>0\).}
In this case, \(w\in\operatorname{int}C\). Hence there exists
\(\varepsilon>0\) such that
\[
    B(w,\varepsilon)\subset C.
\]
Suppose that \(v\in N(w;C)\) and \(v\neq 0\). Choose
\[
    y:=w+\delta v,
    \qquad
    0<\delta<\frac{\varepsilon}{\|v\|}.
\]
Then \(y\in B(w,\varepsilon)\subset C\), while
\[
    \langle v,y-w\rangle
    =
    \langle v,\delta v\rangle
    =
    \delta\|v\|^2>0,
\]
which contradicts the definition of the normal cone. Therefore,
\begin{equation}
    N(w;C)=\{0\}
    \qquad\text{whenever }w_1>0,\;w_2>0.
    \label{7.53}
\end{equation}

\medskip

\noindent
\textit{Case 2: \(w_1=0\) and \(w_2>0\).}
Let
\[
    w=(0,w_2),\qquad w_2>0,
\]
and take
\[
    v=(v_1,v_2)\in N(w;C).
\]
Since \(w_2>0\), we may select points of the form
\[
    y=(0,w_2+\tau)\in C
\]
for every sufficiently small positive or negative \(\tau\). The normal-cone
inequality gives
\[
    \langle v,y-w\rangle
    =
    \langle (v_1,v_2),(0,\tau)\rangle
    =
    v_2\tau
    \leq 0.
\]
Because this must hold for both positive and negative values of \(\tau\),
we obtain
\[
    v_2=0.
\]
Next, choose
\[
    y=(s,w_2)\in C,\qquad s\geq 0.
\]
Then
\[
    \langle v,y-w\rangle
    =
    \langle (v_1,0),(s,0)\rangle
    =
    v_1s
    \leq 0
    \qquad\text{for every }s\geq 0.
\]
It follows that \(v_1\leq 0\). Hence
\begin{equation}
    N(w;C)
    =
    \left\{(v_1,0)\mid v_1\leq 0\right\}
    =
    \left\{\lambda_1(-1,0)\mid\lambda_1\geq 0\right\}
    \label{7.54}
\end{equation}
whenever \(w_1=0\) and \(w_2>0\).

\medskip

\noindent
\textit{Case 3: \(w_1>0\) and \(w_2=0\).}
By an analogous argument, variations in the first coordinate imply
\(v_1=0\), while variations in the nonnegative second coordinate imply
\(v_2\leq 0\). Thus,
\begin{equation}
    N(w;C)
    =
    \left\{(0,v_2)\mid v_2\leq 0\right\}
    =
    \left\{\lambda_2(0,-1)\mid\lambda_2\geq 0\right\}
    \label{7.55}
\end{equation}
whenever \(w_1>0\) and \(w_2=0\).

\medskip

\noindent
\textit{Case 4: \(w_1=w_2=0\).}
Let \(w=(0,0)\) and take
\[
    v=(v_1,v_2)\in N(w;C).
\]
Choosing \(y=(s,0)\in C\), with \(s\geq 0\), gives
\[
    v_1s\leq 0
    \qquad\text{for every }s\geq 0,
\]
and hence \(v_1\leq 0\). Similarly, choosing \(y=(0,s)\in C\)
gives \(v_2\leq 0\). Therefore,
\begin{equation}
    N((0,0);C)
    =
    \left\{(v_1,v_2)\in\mathbb{R}^2
    \ \middle|\ v_1\leq 0,\;v_2\leq 0\right\}.
    \label{7.56}
\end{equation}
Equivalently,
\begin{equation}
    N((0,0);C)
    =
    \left\{
        \lambda_1(-1,0)+\lambda_2(0,-1)
        \ \middle|\ \lambda_1,\lambda_2\geq 0
    \right\}.
    \label{7.57}
\end{equation}

Combining \eqref{7.53}--\eqref{7.57}, we obtain
\begin{equation}
N(w;C)=
\begin{cases}
\{0\},
    & w_1>0,\;w_2>0,\\[1mm]
\{\lambda_1(-1,0)\mid\lambda_1\geq 0\},
    & w_1=0,\;w_2>0,\\[1mm]
\{\lambda_2(0,-1)\mid\lambda_2\geq 0\},
    & w_1>0,\;w_2=0,\\[1mm]
\{\lambda_1(-1,0)+\lambda_2(0,-1)
    \mid\lambda_1,\lambda_2\geq 0\},
    & w_1=w_2=0.
\end{cases}
\label{7.58}
\end{equation}
Consequently,
\begin{equation}
-N(w;C)=
\begin{cases}
\{0\},
    & w_1>0,\;w_2>0,\\[1mm]
\{\lambda_1(1,0)\mid\lambda_1\geq 0\},
    & w_1=0,\;w_2>0,\\[1mm]
\{\lambda_2(0,1)\mid\lambda_2\geq 0\},
    & w_1>0,\;w_2=0,\\[1mm]
\{\lambda_1(1,0)+\lambda_2(0,1)
    \mid\lambda_1,\lambda_2\geq 0\},
    & w_1=w_2=0.
\end{cases}
\label{7.59}
\end{equation}

The sweeping inclusion \eqref{7.48} is therefore equivalent to the existence
of measurable functions
\[
    \eta^1,\eta^2\in L^2([0,1];\mathbb{R}_+)
\]
such that
\begin{equation}
    \dot{x}(t)
    =
    u(t)+\eta^1(t)(1,0)+\eta^2(t)(0,1)
    \qquad\text{for a.e. }t\in[0,1],
    \label{7.60}
\end{equation}
where
\[
    \eta^1(t)(1,0)+\eta^2(t)(0,1)
    \in -N(\dot{x}(t);C).
\]
Writing \eqref{7.60} componentwise gives
\begin{equation}
    \dot{x}_1(t)=u_1(t)+\eta^1(t),
    \qquad
    \dot{x}_2(t)=u_2(t)+\eta^2(t),
    \qquad
    \eta^1(t),\eta^2(t)\geq 0.
    \label{7.61}
\end{equation}

Since the normal cone is evaluated at \(\dot{x}(t)\), feasibility requires
\begin{equation}
    \dot{x}_1(t)\geq 0,
    \qquad
    \dot{x}_2(t)\geq 0
    \qquad\text{for a.e. }t\in[0,1].
    \label{7.62}
\end{equation}
Moreover, the normal-cone formula \eqref{7.59} yields
\begin{equation}
    \dot{x}_1(t)>0
    \quad\Longrightarrow\quad
    \eta^1(t)=0,
    \label{7.63}
\end{equation}
\begin{equation}
    \dot{x}_2(t)>0
    \quad\Longrightarrow\quad
    \eta^2(t)=0,
    \label{7.64}
\end{equation}
and
\begin{equation}
    \eta^1(t)>0
    \quad\Longrightarrow\quad
    \dot{x}_1(t)=0,
    \qquad
    \eta^2(t)>0
    \quad\Longrightarrow\quad
    \dot{x}_2(t)=0.
    \label{7.65}
\end{equation}
Equivalently,
\begin{equation}
\begin{aligned}
    &\dot{x}_1(t)\geq 0,\qquad
      \eta^1(t)\geq 0,\qquad
      \eta^1(t)\dot{x}_1(t)=0,\\
    &\dot{x}_2(t)\geq 0,\qquad
      \eta^2(t)\geq 0,\qquad
      \eta^2(t)\dot{x}_2(t)=0
\end{aligned}
\qquad\text{for a.e. }t\in[0,1].
\label{7.66}
\end{equation}
Indeed, if $\dot x_1(t)>0$ and $\dot x_2(t)>0$, then $\dot x(t)$ lies in the interior of $C$,
and therefore $-N(\dot x(t);C)=\{0\}$, which forces $\eta^1(t)=\eta^2(t)=0$. If $\dot x_1(t)=0$
and $\dot x_2(t)>0$, then the only possible directions in $-N(\dot x(t);C)$ are of the form
$\lambda(1,0)$ with $\lambda\ge0$, so necessarily $\eta^2(t)=0$ while $\eta^1(t)$ may be positive.
Similarly, if $\dot x_1(t)>0$ and $\dot x_2(t)=0$, then $\eta^1(t)=0$ while $\eta^2(t)$ may be
positive. Finally, at the corner $\dot x_1(t)=\dot x_2(t)=0$ both multipliers may be active
simultaneously, corresponding to two active faces of $C$ at the implicit argument
$w(t)=\dot x(t)$.

We can now solve explicitly for the velocities and the reaction
multipliers. If \(u_1(t)\geq 0\), then the choice
\(\eta^1(t)=0\) gives
\[
    \dot{x}_1(t)=u_1(t)\geq 0.
\]
If \(u_1(t)<0\), the choice \(\eta^1(t)=0\) would produce an
inadmissible negative velocity. Hence the reaction must satisfy
\[
    \dot{x}_1(t)=0,
\]
and therefore
\[
    0=u_1(t)+\eta^1(t),
    \qquad
    \eta^1(t)=-u_1(t)>0.
\]
Thus,
\begin{equation}
    \dot{x}_1(t)
    =
    \begin{cases}
        u_1(t),&u_1(t)\geq 0,\\
        0,&u_1(t)<0,
    \end{cases}
    \qquad
    \eta^1(t)
    =
    \begin{cases}
        0,&u_1(t)\geq 0,\\
        -u_1(t),&u_1(t)<0.
    \end{cases}
    \label{7.67}
\end{equation}
The same argument applied to the second coordinate gives
\begin{equation}
    \dot{x}_2(t)
    =
    \begin{cases}
        u_2(t),&u_2(t)\geq 0,\\
        0,&u_2(t)<0,
    \end{cases}
    \qquad
    \eta^2(t)
    =
    \begin{cases}
        0,&u_2(t)\geq 0,\\
        -u_2(t),&u_2(t)<0.
    \end{cases}
    \label{7.68}
\end{equation}
Equivalently,
\begin{equation}
\begin{aligned}
    \dot{x}_1(t)&=\max\{u_1(t),0\},
    &\eta^1(t)&=\max\{-u_1(t),0\},\\
    \dot{x}_2(t)&=\max\{u_2(t),0\},
    &\eta^2(t)&=\max\{-u_2(t),0\}.
\end{aligned}
\label{7.69}
\end{equation}

We next determine an optimal solution of \eqref{7.47}-\eqref{7.50}.
By \eqref{7.62}, both state components are nondecreasing:
\[
    \dot{x}_1(t)\geq 0,
    \qquad
    \dot{x}_2(t)\geq 0.
\]
Therefore,
\begin{equation}
    x_i(1)
    =
    x_i^0+\int_0^1\dot{x}_i(t)\,dt
    \geq x_i^0,
    \qquad i=1,2.
    \label{7.70}
\end{equation}
It follows that every feasible pair satisfies
\begin{equation}
    J[x,u]
    =
    x_1(1)+x_2(1)
    \geq x_1^0+x_2^0.
    \label{7.71}
\end{equation}

Consider now the constant control
\begin{equation}
    \bar{u}(t)=(-1,-1)
    \qquad\text{for every }t\in[0,1].
    \label{7.72}
\end{equation}
Using \eqref{7.69}, the corresponding reaction multipliers are
\begin{equation}
    \bar{\eta}^1(t)=1,
    \qquad
    \bar{\eta}^2(t)=1
    \qquad\text{for every }t\in[0,1].
    \label{7.73}
\end{equation}
Consequently,
\[
\begin{aligned}
    \dot{\bar{x}}(t)
    &=
    (-1,-1)+1(1,0)+1(0,1)=(0,0)
\end{aligned}
\]
for almost every \(t\in[0,1]\). Hence,
\begin{equation}
    \bar{x}(t)=x^0
    \qquad\text{for every }t\in[0,1],
    \label{7.74}
\end{equation}
and
\begin{equation}
    J[\bar{x},\bar{u}]
    =
    x_1^0+x_2^0.
    \label{7.75}
\end{equation}
Thus, the lower bound in \eqref{7.71} is attained, and
\((\bar{x},\bar{u})\) is an optimal pair.

The associated velocity-contact trajectory is
\begin{equation}
    \bar{w}(t)
    =
    \dot{\bar{x}}(t)
    =
    (0,0)
    \qquad\text{for a.e. }t\in[0,1].
    \label{7.76}
\end{equation}

\noindent The set \(C\) can be written in the form
\begin{equation}
    C=
    \left\{
        w\in\mathbb{R}^2
        \ \middle|\
        \langle x_*^1,w\rangle\leq c_1,\;
        \langle x_*^2,w\rangle\leq c_2
    \right\},
    \label{7.77}
\end{equation}
where
\begin{equation}
    x_*^1=(-1,0),
    \qquad
    x_*^2=(0,-1),
    \qquad
    c_1=c_2=0.
    \label{7.78}
\end{equation}
Indeed,
\[
    \langle x_*^1,w\rangle=-w_1\leq 0
    \quad\Longleftrightarrow\quad
    w_1\geq 0,
\]
and
\[
    \langle x_*^2,w\rangle=-w_2\leq 0
    \quad\Longleftrightarrow\quad
    w_2\geq 0.
\]

Since \(\bar{w}(t)=(0,0)\), we have
\[
    \langle x_*^1,\bar{w}(t)\rangle=0=c_1,
    \qquad
    \langle x_*^2,\bar{w}(t)\rangle=0=c_2.
\]
Therefore,
\begin{equation}
    I(\bar{w}(t))=\{1,2\}
    \qquad\text{for every }t\in[0,1],
    \label{7.79}
\end{equation}
and the active-time sets are
\begin{equation}
    E_1=E_2=[0,1].
    \label{7.80}
\end{equation}

We now verify the relations of the necessary optimality system of Theorem~\ref{thm:continuous_mp}.
Note that here $\bar w(t)$ lies on ${\rm bd}\,C$ on the whole interval, i.e., the
interior--boundary switching structure assumed in Theorem~\ref{thm:continuous_mp} degenerates to
$\bar t=0$; the relations below are verified directly.

\medskip

\noindent
\textit{Primal dynamics.}
In this example,
\[
    f(t,x,u)=u,
    \qquad
    h(t,x)=0.
\]
The primal representation in Theorem~\ref{thm:continuous_mp} becomes
\begin{equation}
    \dot{\bar{x}}(t)
    =
    \bar{u}(t)
    -
    \bar{\eta}^1(t)x_*^1
    -
    \bar{\eta}^2(t)x_*^2.
    \label{7.81}
\end{equation}
Using
\[
    \bar{u}(t)=(-1,-1),
    \qquad
    \bar{\eta}^1(t)=\bar{\eta}^2(t)=1,
\]
and
\[
    x_*^1=(-1,0),
    \qquad
    x_*^2=(0,-1),
\]
we obtain
\[
\begin{aligned}
    \dot{\bar{x}}(t)
    &=
    (-1,-1)-1(-1,0)-1(0,-1)\\
    &=(0,0),
\end{aligned}
\]
which agrees with \eqref{7.76}. Furthermore, both multipliers correspond
to active indices because
\[
    I(\bar{w}(t))=\{1,2\}
\]
on the entire interval.

\medskip

\noindent
\textit{Adjoint equation.}
The terminal cost is
\[
    \varphi(x)=x_1+x_2,
\]
and therefore
\begin{equation}
    \nabla\varphi(x)=(1,1).
    \label{7.82}
\end{equation}
Choose
\begin{equation}
    \mu=2,
    \qquad
    p(t)=(-1,-1)
    \quad\text{for every }t\in[0,1],
    \qquad
    \gamma=0.
    \label{7.83}
\end{equation}
Since \(\gamma=0\), the modified adjoint arc is
\begin{equation}
\begin{aligned}
    q(t)
    &=
    p(t)
    -
    \int_{(t,1]}
        x_*^1\,d\gamma^1
    -
    \int_{(t,1]}
        x_*^2\,d\gamma^2\\
    &=p(t)=(-1,-1).
\end{aligned}
\label{7.84}
\end{equation}
Because
\[
    f(t,x,u)=u
    \qquad\text{and}\qquad
    h(t,x)=0,
\]
we have
\begin{equation}
    \nabla_x f(t,x,u)=0,
    \qquad
    \nabla_x h(t,x)=0.
    \label{7.85}
\end{equation}
Thus, the adjoint equation reduces to
\[
    \dot{p}(t)=0,
\]
which is satisfied by the constant arc \(p(t)=(-1,-1)\).

\medskip

\noindent
\textit{Maximum principle.}
Since \(f(t,x,u)=u\),
\begin{equation}
    \nabla_u f(t,x,u)=I.
    \label{7.86}
\end{equation}
The Hamiltonian-type function is therefore
\begin{equation}
\begin{aligned}
    \psi(t,u)
    &:=
    \left\langle
        \nabla_u f(t,\bar{x}(t),u),
        q(t)
    \right\rangle\\
    &=
    \langle u,(-1,-1)\rangle\\
    &=-u_1-u_2.
\end{aligned}
\label{7.87}
\end{equation}
For every \(u=(u_1,u_2)\in U=[-1,1]^2\),
\[
    -u_1-u_2\leq 2.
\]
The upper bound is attained when
\[
    u_1=u_2=-1.
\]
Since \(\bar{u}(t)=(-1,-1)\), we obtain
\begin{equation}
    \psi(t,\bar{u}(t))
    =
    2
    =
    \max_{u\in U}\psi(t,u)
    \qquad\text{for a.e. }t\in[0,1].
    \label{7.88}
\end{equation}
Hence the maximum condition is satisfied.

\medskip

\noindent
\textit{Dynamic complementarity.}
For every \(t\in[0,1]\),
\[
    \bar{w}(t)=(0,0),
\]
and therefore
\begin{equation}
    \langle x_*^1,\bar{w}(t)\rangle
    =
    \langle (-1,0),(0,0)\rangle
    =
    0=c_1,
    \label{7.89}
\end{equation}
while
\begin{equation}
    \langle x_*^2,\bar{w}(t)\rangle
    =
    \langle (0,-1),(0,0)\rangle
    =
    0=c_2.
    \label{7.90}
\end{equation}
Thus both constraints are active on the entire time interval. The
reaction multipliers satisfy
\[
    \bar{\eta}^1(t)=\bar{\eta}^2(t)=1\geq 0,
\]
and both normals are permitted to occur simultaneously at the corner
\(\bar{w}(t)=(0,0)\). Since neither constraint is strictly inactive,
the implications in Theorem~\ref{thm:continuous_mp} associated with strict inactivity impose
no additional restrictions in this example.

\medskip

\noindent
\textit{Terminal conditions.}
At \(t=1\),
\[
    \bar{w}(1)=(0,0),
    \qquad
    I(\bar{w}(1))=\{1,2\}.
\]
Using
\[
    \bar{\eta}^1(1)=\bar{\eta}^2(1)=1,
\]
we obtain
\begin{equation}
\begin{aligned}
    \sum_{j\in I(\bar{w}(1))}
        \bar{\eta}^j(1)x_*^j
    &=
    (-1,0)+(0,-1)\\
    &=(-1,-1).
\end{aligned}
\label{7.91}
\end{equation}
By \eqref{7.58},
\[
    (-1,-1)\in N((0,0);C)
    =
    N(\bar{w}(1);C),
\]
and hence the first terminal condition holds.

Furthermore,
\begin{align}
    -p(1)
    -
    \sum_{j\in I(\bar{w}(1))}
        \bar{\eta}^j(1)x_*^j
    &=
    -(-1,-1)-(-1,-1) \notag\\
    &=(2,2).
    \label{7.92}
\end{align}
On the other hand, using \(\mu=2\) and \eqref{7.82},
\begin{equation}
    \mu\nabla\varphi(\bar{x}(1))
    =
    2(1,1)
    =
    (2,2).
    \label{7.93}
\end{equation}
Therefore,
\begin{equation}
    -p(1)
    -
    \sum_{j\in I(\bar{w}(1))}
        \bar{\eta}^j(1)x_*^j
    =
    \mu\nabla\varphi(\bar{x}(1))
    \in
    \mu\partial\varphi(\bar{x}(1)).
    \label{7.94}
\end{equation}
Thus, both terminal conditions are satisfied.

\medskip

\noindent
\textit{Measure nonatomicity and support.}
Since
\[
    E_1=E_2=[0,1],
\]
the support condition requires
\[
    \operatorname{supp}(\gamma^j)\subset E_j,
    \qquad j=1,2.
\]
For the choice \(\gamma^1=\gamma^2=0\), we have
\begin{equation}
    \operatorname{supp}(\gamma^1)
    =
    \operatorname{supp}(\gamma^2)
    =
    \varnothing
    \subset[0,1].
    \label{7.95}
\end{equation}
The nonatomicity and support requirements therefore hold automatically.

\medskip

\noindent
\textit{Nontriviality.}
The selected multipliers satisfy
\begin{equation}
    (\mu,p,\gamma)
    =
    \bigl(2,(-1,-1),0\bigr)
    \neq(0,(0,0),0).
    \label{7.96}
\end{equation}
Hence the nontriviality condition is satisfied.

We conclude that the optimal pair
\begin{equation}
    \bar{u}(t)=(-1,-1),
    \qquad
    \bar{x}(t)=x^0
    \qquad\text{for every }t\in[0,1],
    \label{7.97}
\end{equation}
together with
\begin{equation}
\begin{aligned}
    &\bar{\eta}^1(t)=\bar{\eta}^2(t)=1,\\
    &\mu=2,\qquad
    p(t)=q(t)=(-1,-1),\qquad
    \gamma^1=\gamma^2=0,
\end{aligned}
\label{7.98}
\end{equation}
satisfies the full necessary optimality system of Theorem~\ref{thm:continuous_mp}.

This example illustrates the multiple-active-face structure of the
implicit sweeping process. The polyhedral set \(C\) has two independent
boundary faces, generated by the outward normals
\[
    x_*^1=(-1,0)
    \qquad\text{and}\qquad
    x_*^2=(0,-1).
\]
Along the selected optimal trajectory,
\[
    \bar{w}(t)=\dot{\bar{x}}(t)=(0,0),
\]
so both faces are active simultaneously for the entire time interval.
The two reaction multipliers exactly cancel the two negative control
components:
\[
    (-1,-1)+1(1,0)+1(0,1)=(0,0).
\]
Thus, the sweeping reaction prevents either component of the velocity
from leaving the admissible quadrant and keeps the state fixed at its
initial value. In contrast to the single-face configuration in
Example~\ref{ex:1}, the dynamics in the present example involve two independent
normal directions acting simultaneously at the corner of the
velocity-constraint set. Note also that, in contrast to the explicit case
\cite{CMN20}, the active-set structure here is determined by the point
$w(t)=A\dot x(t)+h(t,x(t))$ rather than by the state $x(t)$. In particular, multiple active
faces may occur due to the geometry of the velocity-dependent constraint even when the
state $x(t)$ itself is not located on the boundary of the set $C$. Consequently, the
sweeping term may involve several multipliers generated by independent outward normals
at $w(t)$, which is a key feature of the implicit sweeping process and a source of
additional degeneracy in the associated necessary optimality conditions.
\end{example}

The next example shows that, in the implicit setting, the active-face structure is
governed by the contact point $w(t)=A\dot x(t)+h(t,x(t))$ and may involve several faces of
$C$ even when the state itself stays in the interior of $C$.

\begin{example}\label{ex:3}
Consider again the cost and control constraints of Example~\ref{ex:2}:
\begin{equation}
J[x,u]:=x_1(1)+x_2(1)\longrightarrow\min,\qquad u(t)\in U:=[-1,1]\times[-1,1]\ \text{ for a.e. }t\in[0,1],
\label{ex3.1}
\end{equation}
with the same polyhedral cone $C:=\{y\in\mathbb R^2\mid y_1\ge0,\ y_2\ge0\}$, but now choose
\[
A=I,\qquad h(t,x)=-x,
\]
so that the velocity-contact point is
\[
w(t)=A\dot x(t)+h(t,x(t))=\dot x(t)-x(t).
\]
The implicit sweeping inclusion becomes
\begin{equation}
\dot x(t)\in u(t)-N\big(C;\dot x(t)-x(t)\big)\quad\text{for a.e. }t\in[0,1],\qquad x(0)=x^0\in\mathbb R^2 .
\label{ex3.2}
\end{equation}
The normal cone to $C$ is given by \eqref{7.58} of Example~\ref{ex:2}, and $-N(y;C)$ is
generated by $(1,0)$ and $(0,1)$ whenever $y$ lies on the boundary of $C$. Hence any element
$v(t)\in-N(\dot x(t)-x(t);C)$ can be written as
\[
v(t)=\eta^1(t)(1,0)+\eta^2(t)(0,1),\qquad \eta^1(t),\eta^2(t)\ge0,
\]
and substituting into \eqref{ex3.2} yields the representation
\begin{equation}
\dot x_1(t)=u_1(t)+\eta^1(t),\qquad \dot x_2(t)=u_2(t)+\eta^2(t),\qquad \eta^1(t),\eta^2(t)\ge0 .
\label{ex3.3}
\end{equation}
Since the cone $-N(y;C)$ depends on which constraints are active at $y$, the following
complementarity relations are now imposed at the point $y=w(t)=\dot x(t)-x(t)$:
\begin{equation}
\begin{aligned}
&w_1(t)>0\Longrightarrow\eta^1(t)=0,\qquad w_2(t)>0\Longrightarrow\eta^2(t)=0,\\
&\eta^1(t)>0\Longrightarrow w_1(t)=0,\qquad \eta^2(t)>0\Longrightarrow w_2(t)=0
\end{aligned}
\qquad\text{for a.e. }t\in[0,1].
\label{ex3.4}
\end{equation}
Indeed, if $w_1(t)>0$ and $w_2(t)>0$, then $w(t)$ lies in the interior of $C$ and
$-N(w(t);C)=\{0\}$, which forces $\eta^1(t)=\eta^2(t)=0$. If $w_1(t)=0$ and $w_2(t)>0$, then the
only possible directions in $-N(w(t);C)$ are of the form $\lambda(1,0)$ with $\lambda\ge0$, so
necessarily $\eta^2(t)=0$ while $\eta^1(t)$ may be positive. Similarly, if $w_1(t)>0$ and
$w_2(t)=0$, then $\eta^1(t)=0$ while $\eta^2(t)$ may be positive. Finally, at the corner
\[
w_1(t)=w_2(t)=0\quad\Longleftrightarrow\quad \dot x_1(t)=x_1(t),\ \ \dot x_2(t)=x_2(t),
\]
both multipliers may be active simultaneously, corresponding to two active faces of $C$ at
the implicit argument $w(t)=\dot x(t)-x(t)$.

Arguing as in Example~\ref{ex:2} coordinatewise, the feasibility requirement $w_i(t)\ge0$,
i.e., $\dot x_i(t)\ge x_i(t)$, together with \eqref{ex3.3} and \eqref{ex3.4}, yields the
explicit formulas
\begin{equation}
\dot x_i(t)=\max\{u_i(t),x_i(t)\},\qquad \eta^i(t)=\max\{x_i(t)-u_i(t),0\},\qquad i=1,2,
\label{ex3.5}
\end{equation}
for a.e. $t\in[0,1]$. Indeed, if $u_i(t)\ge x_i(t)$, then the choice $\eta^i(t)=0$ gives
$w_i(t)=u_i(t)-x_i(t)\ge0$; if $u_i(t)<x_i(t)$, then $\eta^i(t)=0$ would violate $w_i(t)\ge0$,
so $\eta^i(t)>0$, which forces $w_i(t)=0$, i.e., $\dot x_i(t)=x_i(t)$ and $\eta^i(t)=x_i(t)-u_i(t)$.

Assume now that $x^0\in\mathbb R^2_+$. Since $\dot x_i(t)\ge x_i(t)$, we have
$\frac{d}{dt}\big(e^{-t}x_i(t)\big)\ge0$, and hence every feasible pair satisfies
\begin{equation}
x_i(1)\ge e\,x_i^0,\quad i=1,2,\qquad J[x,u]\ge e\,(x_1^0+x_2^0).
\label{ex3.6}
\end{equation}
Consider the constant control $\bar u(t)=(-1,-1)$. By \eqref{ex3.5}, the corresponding
trajectory solves $\dot{\bar x}_i=\bar x_i$ as long as $\bar x_i(t)\ge-1$, which holds since
$x^0\in\mathbb R^2_+$. Thus
\begin{equation}
\bar x(t)=e^tx^0,\qquad \bar\eta^i(t)=\bar x_i(t)+1=e^tx_i^0+1>0,\qquad \bar w(t)=(0,0)\quad\text{for all }t\in[0,1],
\label{ex3.7}
\end{equation}
and $J[\bar x,\bar u]=e(x_1^0+x_2^0)$, so the lower bound in \eqref{ex3.6} is attained and
$(\bar x,\bar u)$ is optimal. Note that if $x^0\in{\rm int}\,\mathbb R^2_+$, then the state
$\bar x(t)$ stays in the interior of $C$ for all $t$, while the contact point $\bar w(t)$ is
at the corner of $C$, i.e., both faces are active on $[0,1]$.

Let us check the relations of Theorem~\ref{thm:continuous_mp} along this process (as in
Example~\ref{ex:2}, $\bar w(t)\in{\rm bd}\,C$ on the whole interval). Writing
$C=\{w\mid\langle x_*^j,w\rangle\le0,\ j=1,2\}$ with $x_*^1=(-1,0)$ and $x_*^2=(0,-1)$, the LICQ
holds at $\bar w(t)=(0,0)$ and $I(\bar w(t))=\{1,2\}$ for all $t$.

\noindent\textit{Primal dynamics.} By \eqref{ex3.7},
\[
\dot{\bar x}(t)=\bar u(t)-\bar\eta^1(t)x_*^1-\bar\eta^2(t)x_*^2=(-1,-1)+(\bar\eta^1(t),\bar\eta^2(t))=e^tx^0 ,
\]
which is \eqref{6.1} with $\eta^j=\bar\eta^j\in L^2([0,1];\mathbb R_+)$.

\noindent\textit{Adjoint equation.} Here $f(t,x,u)=u$ and $h(t,x)=-x$, so
$\nabla_xf=0$ and $\nabla_xh=-I$. Choosing $\gamma=0$, we have $q=p$ by \eqref{6.3}, and the
adjoint equation \eqref{6.2} reads
\[
\dot p(t)=-\nabla_xh(t,\bar x(t))^*A^*q(t)=q(t)=p(t),
\]
so that $p(t)=e^{t-1}p(1)$.

\noindent\textit{Terminal conditions.} With $\mu=1$, $\eta^1(1)=\eta^2(1)=0$, and
$\nabla\varphi(x)=(1,1)$, condition \eqref{6.7} gives $-p(1)=(1,1)$, i.e.,
\[
p(t)=q(t)=-e^{t-1}(1,1),\qquad t\in[0,1],
\]
while \eqref{6.8} holds trivially since $0\in N(\bar w(1);C)$.

\noindent\textit{Maximum principle.} The Hamiltonian-type function \eqref{6.4} is
\[
\psi(t,u)=\langle u,q(t)\rangle=-e^{t-1}(u_1+u_2),
\]
which is maximized over $U=[-1,1]^2$ exactly at $u=(-1,-1)=\bar u(t)$, so \eqref{6.5} holds.

\noindent\textit{Complementarity, measure, and nontriviality.} Both constraints are active
on $[0,1]$, so \eqref{6.6} imposes no restriction; $\gamma=0$ trivially satisfies (v) and
\eqref{6.11}; and $(\mu,p,\gamma)=(1,-e^{t-1}(1,1),0)\ne(0,0,0)$.

Hence all the relations of Theorem~\ref{thm:continuous_mp} are satisfied. In contrast to
Example~\ref{ex:2}, the optimal state is not stationary: the adjoint arc grows
exponentially through the term $-\nabla_xh^*A^*q$, which is generated by the
state-dependence of the contact point. This illustrates that multiple active faces may occur
due to the geometry of the velocity-dependent constraint even when the state $x(t)$ itself
does not lie on the boundary of $C$; consequently, the sweeping term may involve several
multipliers generated by independent outward normals at $w(t)$, which is a distinctive
feature of the implicit sweeping process and may lead to additional degeneracy in the
associated necessary optimality conditions.
\end{example}

These examples illustrate how implicit sweeping dynamics generate reaction terms that
enforce velocity-dependent constraints. In particular, they show how complementarity
relations arise between the velocity variables (more precisely, the contact point
$w(t)=A\dot x(t)+h(t,x(t))$) and the multiplier functions, and how multiple active faces of
the constraint set may occur due to the geometry of the implicit argument $w(t)$.

\section{Concluding Remarks}\label{sec:concl}

In this paper, we investigated optimal control problems governed by a perturbed implicit
sweeping process with constrained discontinuous controls, where the normal cone to a
moving convex polyhedron is evaluated at the velocity-dependent contact point
$w(t)=A\dot x(t)+h(t,x(t))$. Compared with the explicit case \cite{CMN20}, this setting
requires additional technical considerations due to the dependence of the normal cone on
the velocity term. By adapting the method of discrete approximations and employing
coderivative calculus for the implicit sweeping mapping, we derived necessary optimality
conditions for the discrete-time problems and, by passing to the limit, for the original
continuous-time problem. These conditions generalize the corresponding results obtained in
the explicit case and are formulated in terms of adjoint arcs, measure-valued multipliers,
and a maximum principle of the Hamiltonian type.

The obtained results contribute to the growing literature on optimal control of nonsmooth
dynamical systems governed by sweeping processes, which arise naturally in applications
involving unilateral constraints, mechanical systems with friction, and hysteresis-type
phenomena. Future research directions include nondegenerate and enhanced forms of the
maximum principle under weaker regularity conditions (cf.\ \cite{ArutyunovKaramzin20}),
sufficient optimality conditions, numerical methods for solving implicit sweeping control
problems based on the discrete approximations developed here, and applications to models
arising in mechanics, engineering, and economics.






\section*{Acknowledgements}
Part of the results of this paper are based on the thesis of Yen Dang written at the Department of Mathematics, San Diego State University, under the supervision of Dao Nguyen.

\end{document}